\documentclass[a4,11pt
abstract=on]
{scrartcl}
\usepackage{algorithm, booktabs}
\usepackage{mathtools}
\mathtoolsset{showonlyrefs}
\usepackage{amsfonts}
\usepackage{amsmath}
\usepackage{amssymb}
\usepackage{faktor}
\usepackage{array}
\usepackage{amsthm}
\usepackage{subcaption}
\usepackage{mathdots}
\usepackage{mathrsfs}
\usepackage[noend]{algpseudocode}
\usepackage{graphicx}
\usepackage{color}
\usepackage{epstopdf}
\usepackage{pgf}
\usepackage{scrpage2}
\usepackage{multirow}
\usepackage{listings}
\usepackage{tikz}
\usepackage{scalerel}
\usepackage{hyperref}
\usepackage{float}
\newcommand{\N}{\ensuremath{\mathbb{N}}}

\newcommand{\R}{\ensuremath{\mathbb{R}}}
\newcommand{\C}{\ensuremath{\mathbb{C}}}

\newcommand{\e}{\textnormal{e}}

\newcommand{\norm}[1]{\left\Vert #1\right\Vert}

\newcommand{\tT}{\mathrm{T}}

\newcolumntype{M}[1]{>{\centering\arraybackslash}m{#1}}
\newcolumntype{N}{@{}m{0pt}@{}}

\newcommand{\real}{\mathrm{Re}}
\newcommand{\imag}{\mathrm{Im}}
\newcommand{\Orth}{\mathrm{O}}
\newcommand{\Exp}{\mathrm{Exp}}

\def\SL{\mathrm{SL}}
\def\O{\mathrm{O}}

\newcommand\rst{\mathcal{S}}  
\newcommand{\OB}{{\mathcal{OB}(n)}}

\newcommand\diag{\mathrm{diag}}
\newcommand\tr{\mathrm{tr}}

\newtheorem{thm}{Theorem}[section]
\newtheorem{lemma}[thm]{Lemma}
\newtheorem{remark}[thm]{Remark}

\newtheorem{example}[thm]{Example}
\newtheorem{corollary}[thm]{Corollary}
\newtheorem{proposition}[thm]{Proposition}

\begin{document}
\title{
A New Constrained Optimization Model 
for Solving
the Nonsymmetric Stochastic Inverse Eigenvalue Problem}

\author{
Gabriele Steidl\footnotemark[1]
\and
Maximilian Winkler\footnotemark[2]
}

\maketitle
\footnotetext[1]{Department of Mathematics, TU Berlin,
Stra{\ss}e des 17.~Juni 136, 10587 Berlin
Germany,
\{name\}@math.tu-berlin.de}

\footnotetext[2]{Department of Mathematics,
	Technische Universit\"at Kaiserslautern,
	Paul-Ehrlich-Str.~31, D-67663 Kaiserslautern, Germany,
	\{name\}@mathematik.uni-kl.de}

\begin{abstract}
The stochastic inverse eigenvalue problem aims to reconstruct a stochastic matrix from its spectrum.
While there  exists a large literature on the existence of 
solutions for special settings, 
there are only few numerical solution methods available so far.
Recently, Zhao, Jin and Bai \cite{ZJB2016} proposed a constrained optimization model on the manifold
of so-called isospectral matrices and adapted a modified 
Polak-Ribi\`{e}re-Polyak conjugate gradient method
to the geometry of this manifold.
However, not every stochastic matrix is an isospectral one and
the model in \cite{ZJB2016} is based on the assumption that  for 
each stochastic matrix there  exists
a (possibly different) isospectral, stochastic matrix with the same spectrum.
We are not aware of such a result in the literature, but will see that the claim is at least true for $3 \times 3$ matrices.
In this paper, we suggest to extend the above model by considering matrices which differ from isospectral ones only
by multiplication with a block dia\-gonal matrix with $2 \times 2$ blocks from 
the special linear group $\mathrm{SL}(2)$,
where the number of blocks is given by the number of pairs of complex-conjugate eigenvalues. 
Every stochastic matrix can be written in such a 
form, which was not the case for the form of the isospectral matrices.
We prove that our model has a minimizer and show how the  Polak-Ribi\`{e}re-Polyak conjugate gradient method 
works on the corresponding
more general manifold. 
We demonstrate by numerical examples that the new, more general method performs 
similarly as the one in~\cite{ZJB2016}.
\end{abstract}

\section{Introduction}
%
Stochastic matrices arise in many applications
and have recently got the author's attention in a labeling 
approach in imaging science \cite{APSS2016,BFPS16}.
In this paper, we are interested in the inverse eigenvalue problem for stochastic matrices (StIEP)
i.e., in finding a (nonsymmetric) stochastic matrix with this prescribed spectrum.
Problems of this kind arise, e.g., in applied mechanics, molecular spectroscopy or control theory, see \cite{CGM19,CG05} and the references therein.
StIEP is closely related to the problem of finding a matrix with 
nonnegative entries and prescribed spectrum, called NIEP.
This is due to the fact that for any nonnegative matrix $A$ with positive maximal
eigenvalue $\rho(A)$ and corresponding positive eigenvector $u$, the matrix
$\rho(A)^{-1} \,\diag(u)^{-1} \, A \,  \diag(u)$ is a stochastic 
one, see 
\cite{CG05}. 
There exists a large literature on existence conditions of solutions of StIEP and NIEP for special settings.
Sufficient conditions in case of a real-valued set of eigenvalues were first given in \cite{Suleimanova49}, continued by \cite{Perfect53}. 
Karpelevic \cite{Ka51} completely characterized the set of points $\Theta_n$ which contains the eigenvalues of any stochastic $n\times n$ matrix.
This does not mean that each $n$-tupel of real and complex 
conjugate points from $\Theta_n$ is the spectrum of a 
stochastic 
matrix. 
The complete statement of Karpelevic’s theorem is rather lengthy, see also \cite[Theorem 1.8]{Mi88}.
A shorter formulation can be found in \cite{Ito97} and a more constructive view was recently given in \cite{JP17}. 
For further results, we refer to the survey papers \cite{ELN04,JPMP18} and the monographs \cite{CG05,Xu98}.
However, the general question under which conditions on the eigenvalues a solution of StIEP exists, is still open.

Besides theoretical results, there exist certain numerical approaches to solve StIEP or NIEP. 
In \cite{CMMD14,Lin15},  stochastic matrices with given real-valued eigenvalues fulfilling additional assumptions are constructed explicitly with recursive algorithms. These assumptions are in particular fulfilled if all eigenvalues are positive with maximum $1$. 
In \cite{Orsi06}, an alternating projection - like algorithm was proposed for solving NIEP without further assumptions on the eigenvalue set. 

Another class of algorithms aims to minimize a certain cost function based on matrix factorizations using optimization methods which perform a descend on the underlying matrix manifolds. A first method in this direction was given by Chu \cite{Chu91}, who treated
the symmetric NIEP by solving
for a given vector $\Lambda$ of eigenvalues 
\begin{align} \label{eqn:Chu91MinimizationProblem}
\min_{Q,B} \frac{1}{2} \| B \circ B - Q \, \diag(\Lambda) \, Q^\tT\|_F^2 \quad \text{subject to} \quad Q\in\Orth(n), \ B\in\R^{n,n},
\end{align}
where $\Orth(n)$ is the set of orthogonal matrices, $\circ$ denotes the componentwise product and $\| \cdot\|_F$ is the Frobenius norm. 
The above problem is manifold constrained in the variable $Q$ and a gradient flow with Riemannian gradient on $\Orth(n)$ was applied.
Extending the idea to not necessarily symmetric nonnegative matrices, a similar approach was developed later in \cite{CG98}. 
After creating a real Jordan form $J(\Lambda)$ with the given real and imaginary parts of the eigenvalues, a gradient flow method for minimizing
\begin{align}\label{eqn:Chu98MinimizationProblem}
\min_{T,B}  \frac{1}{2} \| B \circ B - T \, J(\Lambda) \, T^{-1}\|_F^2 \quad 
\text{subject to} \quad T \in \mathrm{GL}(n,\R), \ B\in\R^{n,n},
\end{align}
where $\mathrm{GL}(n,\R)$ denotes the (open) manifold of invertible $n \times n$ matrices,
was proposed, where a singular value decomposition of the matrix $T$ is used 
to overcome the instabilities due to the matrix inverse. 
Apart from promising numerical results, the authors report several problems: the existence of a minimizer 
is unclear, 
there is no guarantee for the matrix $T$ to stay bounded during the algorithm, and for several choices of initial values it appears that $T$ converges to a singular matrix. 
Recently, a model for solving StIEP also for nonsymmetric matrices
together with a
geometric optimization method, namely a conjugate gradient descent algorithm in 
a modified version of Polak-Ribi\`{e}re-Polyak,  
was proposed by Zhao, Jin and Bai \cite{ZJB2016}. 
The minimization is basically done over the product manifold of the orthogonal matrices and matrices with rows on the unit
sphere of $\mathbb R^n$.
The model is based on the assumption that StIEP has a solution if and only if the set of so-called isospectral matrices 
$\mathrm{Iso}(\Lambda)$ associated
to the given vector of eigenvalues $\Lambda$ contains a stochastic matrix. While it is clear that the problem has a solution if 
$\mathrm{Iso}(\Lambda)$ contains a stochastic matrix, the opposite direction is to the best of our knowledge not proved so far.
In this paper, we propose an extension of the model in  \cite{ZJB2016} by  special blockdiagonal matrices whose $2\times2$ blocks
are special linear matrices and
correspond to the complex conjugate eigenvalue pairs of the matrix.
Our model has the advantage that every stochastic matrix can be written in the novel form which was not the case for 
isospectral matrices.
We also use a geometric conjugate gradient algorithm for the minimization which performs on the manifold extended by the product of 
$\mathrm{SL}(2)$ matrices. In contrast to \cite{ZJB2016}, the Riemannian inner product of this manifold 
depends on the respective tangent space.
Our model appears to be slightly slower than the model from \cite{ZJB2016}, 
which is not surprising, since we have to minimize over more parameters.
In general, StIEP remains a severely ill posed problem and the 
performance of all algorithms heavily depends on the
distribution of the given eigenvalues.

The outline of this paper is as follows: 
in Section \ref{sec:prelim} we provide the preliminaries 
and motivate our model based on the fact that it is not clear if 
	for any stochastic matrix there exists
a (possibly different) isospectral, stochastic matrix with the same spectrum.
Using the discrete Fourier transform,  we prove that the above conjecture is at least correct for stochastic $3 \times 3$ matrices in Section \ref{sec:33}.
Here the reader may get an idea why its proof appears to be hard for arbitrary matrix sizes.
In Section \ref{sec:model}, we introduce our novel model based on those presented by Zhao, Jin and Bai \cite{ZJB2016}.
We prove that the
new cost function has compact level sets which implies the existence of a (global) minimizer.
We give a short introduction to conjugate gradient algorithms on manifolds in Section~\ref{sec:algs} 
and compute the required expressions for our special manifold in Section~\ref{sec:specmani}. 
We consider an additional parameter update step in line search which seems to 
accelerate the performance slightly for the model in \cite{ZJB2016} and is also 
needed in the convergence proof.
Theorem \ref{thm:SameConvergenceTheorem} states a convergence result for the conjugate gradient method on our manifold.
Since many, but not all  parts of the proof follow the ideas in \cite{ZJB2016}, we swap it to the appendix.
In Section~\ref{sec:numerics}, we demonstrate by numerical experiments that our method 
is competitive with the one in \cite{ZJB2016}, but slightly costlier.

\section{Preliminaries} \label{sec:prelim}
By Schur's theorem we know that for every matrix $A \in \mathbb C^{n,n}$ 
there exists a unitary matrix $U \in\mathbb{C}^{n,n}$ such that 
$A = U \, B \, U^*$,
where $B$ is an upper triangular matrix having the eigenvalues of $A$ on its diagonal.
For real-valued matrices this result can be specified by the following theorem, 
see, e.g., \cite[Theorems 2.3.1, 2.3.4]{HJ12}. 
Recall that the non-real eigenvalues of a real-valued matrix appear in conjugate pairs.
By  $\Orth(n)$ we denote the set of orthogonal $n\times n$ matrices.

\begin{thm}\label{thm:Schur}
Let $A \in\mathbb{R}^{n,n}$ be an arbitrary matrix with real eigenvalues
$\lambda_1, \ldots, \lambda_s$ 
and complex eigenvalues 
$\lambda_{s+1},\ldots,\lambda_n \in \mathbb C \setminus \mathbb R$ 
appearing in conjugate pairs $\lambda_{s+2j} = \bar \lambda_{s+2j-1}$, $j=1,\ldots, t$,  $t = \frac{n-s}{2}$.
Let $\Lambda \coloneqq (\lambda_1, \ldots, \lambda_n)^\tT$.
Then the following holds true:
\begin{itemize}
\item[\textup{i)}] There exists a real-valued, invertible matrix $T 
\in\mathbb{R}^{n,n}$ such that 
$$
A = T  \left( D(\Lambda) + V \right) T^{-1},
$$
where 
\begin{align*}
D(\Lambda) \coloneqq
\begin{pmatrix}
\lambda_1 & 0 & \hdots & \hdots & \hdots & \hdots &  \hdots & 0  \\
0 & \lambda_2 & 0 & \hdots & \hdots  & \hdots & \hdots & 0 \\
0 & 0 & \ddots & 0 & \hdots & \hdots & \hdots &0\\
0 & 0 & 0 & \lambda_{s} & 0 & \hdots & \hdots & 0  \\
0 & 0 & 0 & 0 & \lambda_{1}^{[2]} & 0  & \hdots & 0 \\
  &  &  & &  & \ddots & & &   \\
  &  &  & &  & & \ddots & &   \\
0 & 0 & 0 & 0 & 0 & 0 &  0 & \lambda_{t}^{[2]} 
\end{pmatrix}
\end{align*}
with
\begin{align*}
\lambda_j^{[2]} \coloneqq \left(
\begin{array}{rr}
\real(\lambda_{s+2j-1}) & \imag(\lambda_{s+2j-1}) \\
-\imag(\lambda_{s+2j-1}) & \real(\lambda_{s+2j-1})
\end{array} \right), \ \ j = 1,\ldots,t,
\end{align*}
and $V \in \mathcal{V}_t$. Here $\mathcal{V}_t$ 
denotes the set of upper triangular matrices with zeros on the diagonal and 
$v_{s+2j-1,s+2j} = 0$ for all $j=1,\ldots,t$.
If $A$ has only real eigenvalues, then $T$ can be chosen as an orthogonal 
matrix.
\item[\textup{ii)}] 
There exists a matrix $Q \in \Orth(n)$ such that 
$$A = Q  \left( \tilde D(\Lambda) + V \right) Q^\tT$$ 
with
\begin{align*}
\tilde D(\Lambda) \coloneqq
\begin{pmatrix}
\lambda_1 & 0 & \hdots & \hdots & \hdots & \hdots &  \hdots & 0  \\
0 & \lambda_2 & 0 & \hdots & \hdots  & \hdots & \hdots & 0 \\
0 & 0 & \ddots & 0 & \hdots & \hdots & \hdots &0\\
0 & 0 & 0 & \lambda_{s} & 0& \hdots & \hdots & 0  \\
0 & 0 & 0 & 0 & \mu_{1}^{[2]} & 0  & \hdots & 0 \\
 &  &  & &  & \ddots & & &   \\
  &  &  & &  & & \ddots & &   \\
0 & 0 & 0 & 0 & 0 & 0 &  0 & \mu_{t}^{[2]} 
\end{pmatrix}
\end{align*}
where $\mu_{j}^{[2]} \in \mathbb R^{2,2}$ is a matrix with eigenvalues $\bar \lambda_{s+2j},  \lambda_{s+2j}$, $j=1,\ldots,t$
and $V \in \mathcal{V}_t$.
\end{itemize}
\end{thm}

It is in general not possible to choose matrices $\mu_{j}^{[2]}$ in Part ii) 
of the theorem of the special form $\lambda_{j}^{[2]}$ from Part i), as the Example \ref{ex:33} in the next section shows.
For a given vector 
\begin{equation}\label{self-adj}
\Lambda \coloneqq (\lambda_1,\ldots,\lambda_n )^\tT, \; \lambda_1,\ldots,\lambda_s \in \mathbb R, \; 
\lambda_{s+1},\ldots,\lambda_n \in \mathbb C \setminus \mathbb R,\;
\lambda_{s+2j}= \bar \lambda_{s+2j-1}, j=1,\ldots,t,
\end{equation}
where $t = (n-s)/2$,
let $D(\Lambda)$ and $\mathcal{V}_t$ be defined as in Theorem \ref{thm:Schur}i) and
\begin{align} \label{isospect}
\mathrm{Iso}(\Lambda) \coloneqq \{A \coloneqq Q (D(\Lambda) + V) Q^\tT: Q \in \Orth(n), \, V \in \mathcal{V}_t \}.
\end{align}
The set $\mathrm{Iso}(\Lambda)$ is known as set of \emph{isospectral matrices} associated with $\Lambda$.
In the following,  we will use $\Lambda$ likewise as set of eigenvalues or as vector containing the eigenvalues,
where the meaning becomes always clear from the context. 
\newline

We are interested in the set of stochastic $n \times n$ matrices
\begin{align*}
\rst(n) &\coloneqq \{ A \in \mathbb{R}_{\ge 0}^{n,n}: A  1_n =  1_n \}\\
&= \{S \circ S: \mathrm{diag} (S S^\tT) = I_n,\; S \in \mathbb R^{n,n} \},
\end{align*}
where $\circ$ denotes the componentwise product and for $X\in\R^{n,n}$, 
$\diag(X)$ denotes the diagonal 
matrix with same diagonal entries.
For $A \in \rst(n)$ we have $A 1_n = 1_n$, 
so that 1 is an eigenvalue of $A$ with eigenvector $1_n$.
Moreover, by the Perron-Frobenius theorem \cite[Theorem 5.2.1]{Xu98}, all 
eigenvalues of $A$ have absolute values not larger than 1.

We call a vector $\Lambda$ whose non-real components appear in conjugate pairs, \emph{self-conjugate}.
Given a self-conjugate vector $\Lambda$ we are interested in finding a stochastic matrix having the components of $\Lambda$ as eigenvalues.
As already mentioned in the introduction, this problem has no solution for general self-conjugate vectors $\Lambda$.
In this paper, we examine the following problem:
\begin{center}
\begin{tabular}{ll}
(StIEP)\hspace{0.7cm}  &Given a vector $\Lambda$ whose entries are the 
eigenvalues of a stochastic \\  
                       &matrix, find a stochastic matrix with these eigenvalues.
\end{tabular}
\end{center}

\begin{remark}\label{rem:wrong}
In \textup{\cite{ZJB2016}}, the authors claimed that \textup{(StIEP)} 
has a solution if and only if $\mathrm{Iso}(\Lambda) \cap \rst(n) \not = \emptyset$.
The authors develop a numerical algorithm to solve \textup{(StIEP)} based on 
this claim.
Clearly, if  $\mathrm{Iso}(\Lambda) \cap \rst(n) \not = \emptyset$, then 
\textup{(StIEP)} has a solution.
However, we have not found a reference that the converse is also true.  
In general, we do not know if it is possible that \textup{(StIEP)} has a 
solution, 
but there is no solution with decomposition \eqref{isospect}, i.e. $\mathrm{Iso}(\Lambda) \cap \rst(n) = \emptyset$.
\end{remark}

At least for stochastic $3 \times 3$ matrices there are some results in this direction which
we summarize and partially prove in the next section.

\section{Stochastic $3 \times 3$ Matrices} \label{sec:33}
%
First, we give an example of a stochastic $3 \times 3$ matrix which cannot be decomposed as in Theorem \ref{thm:Schur}ii)
with $\mu_1^{[2]}  = \lambda_1^{[2]}$. 

\begin{example} \label{ex:33}
Let $A\in \rst(3) \cap \mathrm{Iso}(\Lambda)$ with eigenvalue vector
$\Lambda = (1,\lambda,\bar \lambda)^\tT$, $\lambda = \lambda_R + i \lambda_I$, $\lambda_I \not = 0$.
Then the matrix $Q \in \Orth(3)$ in the isospectral decomposition 
\textup{\eqref{isospect}} of $A$ 
has the normed eigenvector of $A$ belonging to the eigenvalue $1$ as first column, i.e.,
$\pm \frac{1}{\sqrt{3}} 1_3$. 
We denote the second column of $Q$ by
$\sqrt{3} (q_1,q_2,q_3)^\tT$, where $q_1+q_2+q_3 = 0$ and $q_1^2 + q_2^2 + q_3^2 = 1$.
Using the vector product of the first two columns of $Q \in \Orth(3)$, we conclude that, up to sign changes of its columns,
$Q$ must have the form 
\begin{align}\label{form_p}
Q = \frac{1}{\sqrt{3}} 
\left( \begin{array}{rrr}
1&\sqrt{3}  q_1 &  q_3-q_2 \\ 
1&\sqrt{3}  q_2 &  q_1-q_3 \\ 
1&\sqrt{3}  q_3 &  q_2-q_1
\end{array} \right)
= 
\frac{1}{\sqrt{3}} 
\left( \begin{array}{ccc}
1&\sqrt{3}  q_1 &  -q_1-2q_2 \\ 
1&\sqrt{3}  q_2 &  2q_1+q_2 \\ 
1&-\sqrt{3} (q_1+q_2) &  q_2-q_1
\end{array} \right)
\end{align} 
with 
\begin{equation}
q_1^2 + q_2^2 + q_1 q_2 = \frac12.
\end{equation}
Note that the last equation has a real solution if and only if  $q_2^2 \le \frac23$.
Since $A \in \mathrm{Iso}(\Lambda)$, it holds
\begin{equation}\label{a}
Q^\tT A Q = 
\left( \begin{array}{rrr}
1 & u & v \\ 0 & \lambda_R & \lambda_I \\ 0 & -\lambda_I & \lambda_R
\end{array} \right)
\end{equation}
for some $u,v \in \mathbb R$, so that
\begin{equation}\label{aa}
\sqrt{3} A Q =  \left( 
\begin{array}{rrr}
1 &  u + \sqrt{3}  \lambda_R q_1- \lambda_I(q_3-q_2)  & v + \sqrt{3} \lambda_I q_1 + \lambda_R (q_3-q_2)   \\ 
1 &  u + \sqrt{3}  \lambda_R q_2- \lambda_I(q_1-q_3)  & v + \sqrt{3} \lambda_I q_2  + \lambda_R (q_1-q_3)  \\ 
1 &  u + \sqrt{3}  \lambda_R q_3- \lambda_I(q_2-q_1)  & v + \sqrt{3} \lambda_I q_3  + \lambda_R (q_2-q_1) 
\end{array} 
\right).
\end{equation}
We want to show that the stochastic matrix
\begin{align} \label{Counterexample3x3}
A = \left( 
\begin{array}{ccc}
\frac{1}{2} & \frac{1}{2} & 0 \\
\frac{1}{3} & \frac{1}{3} & \frac{1}{3} \\
1 & 0 & 0
\end{array}
\right)
\end{align}
does not have a decomposition \eqref{isospect}.
The matrix $A$ has the eigenvalues $1$ and $\frac{1}{12} (-1 \pm \sqrt{23}i)$. 
Assume in the contrary that $A$ has a decomposition  \eqref{isospect}.
Comparison of the second and third columns of the matrices in \eqref{aa} gives
\begin{align} \label{e1}
\frac{\sqrt{3}}{2} (q_1+q_2) &= u + \sqrt{3} \lambda_R q_1 - \lambda_I(q_3-q_2) ,\\
0&=u + \sqrt{3}  \lambda_R q_2 - \lambda_I(q_1-q_3) , \label{e2}\\
\sqrt{3} q_1 &=u + \sqrt{3}  \lambda_R q_3 - \lambda_I(q_2-q_1) ,\label{e3}\\
\frac12(q_1-q_2) &= v + \sqrt{3}  \lambda_I q_1 + \lambda_R(q_3-q_2) ,\label{e4}\\
0&=v + \sqrt{3}  \lambda_ Iq_2 + \lambda_R(q_1-q_3) ,\label{e5}\\
q_3-q_2 &=  v + \sqrt{3} \lambda_I q_3 + \lambda_R(q_2-q_1) .\label{e6}
\end{align}
Replacing $u$ and $v$ by applying the second and fourth equation and $q_3 = -(q_1+q_2)$, we get
\begin{align*}
0 &= q_1\left(-\frac{1}{2} +  \lambda_R + \sqrt{3} \lambda_I \right) + q_2 \left(-\frac{1}{2} -  \lambda_R + \sqrt{3} \lambda_I\right),\\
0 &= q_1\left(-1 -  \lambda_R + \sqrt{3} \lambda_I \right) - 2 q_2   
\lambda_R,\\
0 &= q_1\left(-\frac12 - 3 \lambda_R + \sqrt{3} \lambda_I\right) + q_2 \left(\frac12 - 3 \lambda_R - \sqrt{3} \lambda_I\right),\\
0 &= q_1\left(1-3\lambda_R - \sqrt{3} \lambda_I\right) + q_2\left(2 - 2 \sqrt{3} \lambda_I\right).
\end{align*}
For $\lambda_R = -\frac{1}{12}$ and $\lambda_I = \frac{\sqrt{23}}{12}$ or $\lambda_I = -\frac{\sqrt{23}}{12}$,
this linear system of equations has only the trivial solution $q_1 = q_2 = 0$. Since then $q_3$ is also zero,
this contradicts the orthogonality of $Q$. 

Changing signs of the columns of $Q$ would lead to three new cases in \eqref{e1}-\eqref{e6},
namely the changes
\begin{align*}
&Q \rightarrow Q \mathrm{diag}(1,-1,1) : \quad u \rightarrow -u, \; \lambda_I \rightarrow -\lambda_I,\\
&Q \rightarrow Q \mathrm{diag}(1,-1,1) : \quad v \rightarrow -v, \; \lambda_I \rightarrow -\lambda_I,\\
&Q \rightarrow Q \mathrm{diag}(1,-1,-1): \, u \rightarrow -u, v \; \rightarrow -v.
\end{align*}
In all three cases, a simple variable substitution leads to the same final system of equations. 
\hfill $\Box$
\end{example}
\medskip

Next, let us characterize the spectra of  stochastic $3 \times 3$ matrices.
The facts connected in the following theorem can be found in \cite{LL78}.

\begin{thm} \label{LoewyLondon}
A set $\Lambda = \{ \lambda_1, \lambda_2, \lambda_3 \}$ with $\Lambda = \bar \Lambda$ is the spectrum of a matrix 
$A \in \mathbb{R}_{\ge 0}^{3,3}$ 
if and only if
\begin{align}
\max_{1\leq k \leq 3}|\lambda_k| 
& \in\Lambda, \label{LL1} \\
\mathrm{tr} (A) = \lambda_1+\lambda_2+\lambda_3
&\geq 0, \label{LL2} \\
(\lambda_1+\lambda_2+\lambda_3) ^2 
&\leq 3 (\lambda_1^2+\lambda_2^2+\lambda_3^2). \label{LL3} 
\end{align} 
\end{thm}

Note that for real-valued $\lambda_k$, $k=1,2,3$ the condition \eqref{LL3} is automatically fulfilled by the
relation between arithmetic and geometric means. We further need the following auxiliary lemma.

\begin{lemma}\label{lem:stochastic_nn} \textup{\cite[Lemma 5.3.2]{Xu98}}
Let $A\in\mathbb{R}^{n,n}_{\geq 0}$ with eigenvalues $\{ \lambda_1, ..., \lambda_n \}$ and spectral radius $\lambda_1$. 
Then there exists $B\in\mathbb{R}^{n,n}_{\geq 0}$ with same eigenvalues such that 
\begin{align*} 
B 1_n=\lambda_1 \, 1_n. 
\end{align*} 
\end{lemma}

Then we can simply deduce the following corollary. 

\begin{corollary} \label{CorNonnegToStoch}
\begin{itemize}
\item[\textup{i)}] Let $\Lambda = \{1,\lambda_2,\lambda_3\}$ with $\lambda_2, 
\lambda_3 \in\mathbb{R}$. 
Then there exists $A\in \rst(3)$ with spectrum $\Lambda$ if and only if  
$\lambda_2,\lambda_3 \in [-1,1]$ and $\lambda_2 + \lambda_3 \geq -1$.
\item[\textup{ii)}] Let $\Lambda = \{1,\lambda, \bar \lambda\}$ with $\lambda 
\not \in\mathbb{R}$. 
Then there exists $A\in \rst(3)$ with spectrum $\Lambda$ if and only if $\lambda \in \Theta_3$, where
\begin{align} \label{set}
\Theta_{3} 
&\coloneqq\bigl\{\lambda_R + i \lambda_I: \, \lambda_R \in [-\tfrac12,1], \, 
(\lambda_R -1)^2 \ge 3 \lambda_I^2\bigr\} \\
 &=  \mathrm{conv} \bigl\{1, \theta,\theta^2 \bigr\}, \quad \theta = \e^{-2\pi 
 i/3} = -\frac12 - \frac{\sqrt{3}}{2} i. 
\end{align}
and $\mathrm{conv}$ denotes the convex hull. 
\end{itemize}
\end{corollary}

The set $\Theta_3 \cup [-1,1]$ is depicted in Fig. \ref{Theta3}.

  \begin{figure}[h]
	\centering
	\includegraphics[width=55mm, scale=0.5]{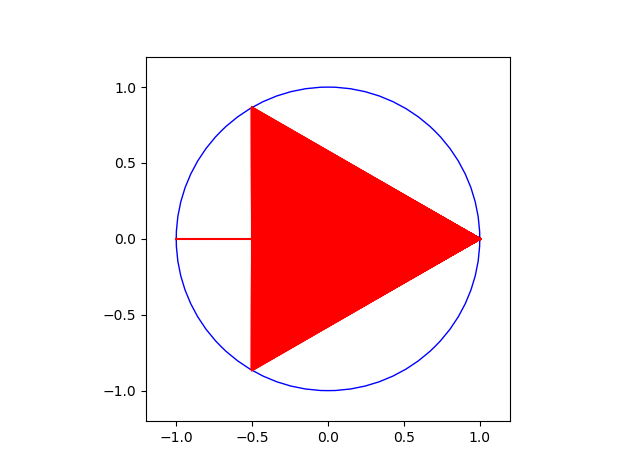}
	\caption{The set $\Theta_3 \cup [-1,1]$.}
	\label{Theta3}
	\end{figure}

\begin{proof}
i)   By Theorem \ref{LoewyLondon}, there exists a matrix 
$A\in\mathbb{R}_{\ge 0}^{3,3}$ with spectrum $\Lambda$ 
if and only if the two conditions in Part i) are fulfilled.
Then we obtain the assertion by Lemma \ref{lem:stochastic_nn}.
 
ii) Let $\lambda_2 = \lambda_R + i \lambda_I$,  $\lambda_I \neq 0$. 
We prove that \eqref{LL1}-\eqref{LL3} hold true if and only if $\lambda_2 \in \Theta_3$. 
Then, the assertion follows again by Lemma \ref{lem:stochastic_nn}. 
 
For our setting we have that 
\eqref{LL1} is equivalent to $|\lambda_2|\leq 1$,
\eqref{LL2} to $\lambda_R \ge -\frac12$, and
\eqref{LL3} to
$$
(1+2\lambda_R )^2 \leq 3(1+ (\lambda_R +i \lambda_I)^2 + (\lambda_R -i \lambda_I)^2) = 6\lambda_R^2-6\lambda_I^2+3,  
$$
i.e. $\lambda_I^2 \leq \frac13 (\lambda_R-1)^2$.
This yields the assertion.
\end{proof}

We will use the set of circulant $n \times n$ matrices 
$$
\bigl\{\mathrm{circ} (a) = a_0 I_n + a_1 P_n +  \ldots + a_{n-1} P_n^{n-1}: a = 
(a_0,\ldots,a_{n-1})^\tT \in \mathbb R^n\bigr\}
$$
with the $n$-th shift matrix
$$
P_n \coloneqq \left( \begin{array}{cc}
0_{n-1} & 1  \\ 
I_{n-1} & 0_{n-1}^\tT \\ 
\end{array} \right).
$$
Clearly, $P_n$ is double stochastic. Note that the double stochastic matrices are the convex hull of the permutation matrices due to the Theorem of Birkhoff and Von Neumann.
The circulant matrices can be diagonalized by the $n$-th Fourier matrix
$F_n \coloneqq (\e^{-2\pi i jk/n})_{j,k=0}^{n-1}$, i.e.,
\begin{equation}\label{circ_d}
\mathrm{circ} (a) = F_n^{-1} \, \mathrm{diag} (F_n a) \, F_n
\end{equation}
see \cite{PPST19}. Note that $F_n^{-1} = \frac{1}{n} \bar F_n = \frac{1}{n} \bar F_n^\tT$.
Then it is easy to check the following lemma.

\begin{lemma} \label{fourier}
Let $n \in \mathbb N$ be odd, $m = (n-1)/2$ and $\Lambda = (1,\lambda_1, \ldots, \lambda_m,\bar \lambda_m, \ldots, \bar \lambda_1)^\tT$.
Then the eigenvalues of
$
B \coloneqq \mathrm{circ} \left( F_n^{-1} \Lambda \right) 
$
are the components of $\Lambda$, $B$ is real-valued and the rows of $B$ sum up to 1. 
\end{lemma}

\begin{proof}
By \eqref{circ_d} the matrix $B$ has the entries of $\Lambda$ as eigenvalues.
Further
$$ F_n^{-1} \Lambda  
= \frac1n \left( 1 + \sum_{j=1}^m \left( \e^{2\pi i jk/n} \lambda_j +  \e^{- 2\pi i jk/n} \bar \lambda_j \right)  \right)_{k=0}^{n-1}
$$
is real-valued and 
$$
1_n^\tT F_n^{-1} \Lambda 
= \frac1n 
\Big( 
\big(
\sum_{k=0}^{n-1} \e^{2\pi i jk/n} 
\big)_{j=0}^{n-1}
\Big)^\tT \Lambda = (1,0,\ldots ,0)^\tT \Lambda = 1.
$$
\end{proof}

Now we can prove the following proposition.

\begin{proposition}\label{PropOrthDim3}
For any $A \in \mathcal{S}(3)$, there exists a matrix $B \in \mathcal{S}(3)$ with the same eigenvalues 
which can be decomposed as in  \eqref{isospect}.
If $A \in \mathcal{S}(3)$ has an eigenvalue $\lambda = \lambda_R + i \lambda_I$, $\lambda_I \not = 0$, 
then the matrix
$$
B = \mathrm{circ} \bigl(F_3^{-1} (1,\lambda,\bar \lambda)^\tT\bigr) \in \rst(3),
$$
has the same eigenvalues as $A$ and possesses the decomposition
$$
B = 
	Q \left( \begin{array}{rrr}
1 & 0&0  \\ 
0 & \lambda_R & \lambda_I \\ 
0 & -\lambda_I& \lambda_R
\end{array} \right) 
Q^\tT
$$
where $Q$ is any matrix of the form \eqref{form_p}. 
\end{proposition}

\begin{proof}
1. If the eigenvalues of $A$ are real-valued, then the assertion follows from Theorem \ref{thm:Schur} i).
\\
2. Assume that the eigenvalues of $A$ are the entries of the vector 
$\Lambda = (1,\lambda,\bar \lambda)$.
By Lemma \ref{fourier}, the matrix $B$ is real-valued, has the same eigenvalues as $A$ and its rows sum up to 1.
Moreover, we have by Corollary \ref{CorNonnegToStoch}ii) that
$$
b \coloneqq F_3^{-1} \begin{pmatrix} 1\\\lambda\\ \bar \lambda\end{pmatrix} = 
\frac13 
\left(
\begin{array}{lll}
1+\lambda + \bar \lambda\\
1 + \bar \theta \lambda +  \theta \bar \lambda\\
1 + \theta \lambda + \bar \theta  \bar \lambda
\end{array} 
\right)
 = 
\frac13 
\left(
\begin{array}{lll}
1+2 \lambda_R\\
1 - \lambda_R - \sqrt{3} \lambda_I\\
1 - \lambda_R + \sqrt{3} \lambda_I
\end{array} 
\right)
 \in \mathbb R_{\ge 0}^3.
$$
By straightforward computation, we obtain for $Q \in \Orth(3)$ as in \eqref{form_p} that
$$
Q^\tT P_3 Q = \left(
\begin{array}{rrr}
1&0&0\\
0& -\frac12& -\frac{\sqrt{3}}{2}\\
0& \frac{\sqrt{3}}{2}&-\frac12
\end{array} 
\right).
$$
so that
$$
Q^\tT B Q = 
\left(
\begin{array}{ccc}
b_0+b_1+b_2&0&0\\
0& b_0 -\frac12(b_1+b_2)& -\frac{\sqrt{3}}{2} (b_1-b_2)\\
0& \frac{\sqrt{3}}{2}(b_1-b_2)&b_0 -\frac12(b_1+b_2)
\end{array} 
\right)
=
\left( \begin{array}{rrr}
1 & 0&0  \\ 
0 & \lambda_R & \lambda_I \\ 
0 & -\lambda_I& \lambda_R
\end{array} \right) .
$$
\end{proof}

The matrix $B$ in the above proposition is bistochastic.
Finally, let us give an example of a stochastic matrix such that there does not exist a bistochastic matrix with the same eigenvalues.

\begin{example}\label{ex:[-1,0,1]}
By Theorem \textup{\ref{LoewyLondon}} there exist stochastic matrices with 
eigenvalues $\{1,0,  \\ -1\}$, for example
\begin{align*}
\left( \begin{array}{ccc}
0 & 1 & 0 \\
0 & 0 & 1 \\
0 & 1 & 0
\end{array} \right) \quad \mathrm{or} \quad 
\left( \begin{array}{ccc}
0 & 1 & 0 \\
1 & 0 & 0 \\
u & v & 0
\end{array} \right)
\end{align*}
with $u,v\geq 0$ and $u+v=1$.

However, there is no bistochastic matrix with these eigenvalues.
Otherwise, it would have trace 0 and is therefore of the form 
\begin{align*}
B = \left(
\begin{array}{ccc}
0 & a & 1-a \\
1-a & 0 & a \\
a & 1-a & 0
\end{array}
\right), \qquad a \in [0,1].
\end{align*}
The characteristic polynomial of $B$ is 
$-\lambda^3 + a^3 + (1-a)^3 + 3a\lambda(1-a)$. 
Since $0$ is an eigenvalue, we get $0 = a^3 + (1-a)^3=1-3a+3a^2$. This quadratic equation has 
no real-valued solution, which contradicts our assumption.
\hfill $\Box$
\end{example}

\section{Old and New Model} \label{sec:model}
%
Throughout this section, let 
$\lambda_1,\ldots,\lambda_s \in \mathbb R$ 
and conjugate pairs of complex numbers
$\lambda_{s+1}, \bar \lambda_{s+1}, \ldots, \lambda_n, \bar \lambda_n$ 
be given. 
Set $t \coloneqq \frac{n-s}{2}$. 
We define the vector $\Lambda$ whose components are these numbers as in \eqref{self-adj}
and $D(\Lambda)$ and $\mathcal{V}_t$ as in Theorem \ref{thm:Schur}i).

Then the authors of \cite{ZJB2016}
considered the optimization problem
\begin{align*}
\min_{S,Q,V} \ \widetilde{F}(S,Q,V) \quad \mathrm{ subject \; to} \quad  (S,Q,V)\in \widetilde{\mathcal{M}}_t
\end{align*}
where
\begin{align} \label{problem:old}
\widetilde{F}(S,Q,V) 
\coloneqq 
\frac{1}{2} \| S \circ S - Q(D(\Lambda)+V)Q^\tT\|_F^2, 
\end{align}
and
$$
\widetilde{\mathcal{M}}_t \coloneqq \mathcal{OB}(n) \times \Orth(n) \times \mathcal V_t,
$$
with 
\begin{align*}
\mathcal{OB}(n) \coloneqq \{ S\in \mathbb{R}^{n,n}: \mathrm{diag}(SS^\tT)=1_n\}.
\end{align*}
Clearly, 
$\rst(n) = \{ S \circ S: S\in \OB \}$.
By \cite{ZJB2016}, the level sets of $\widetilde{F}$ are compact, 
so that there exists a minimizer of $\widetilde{F}$. 

However, as emphasized in Remark \ref{rem:wrong}, even if $\lambda_k$, $k=1,\ldots,n$, are the eigenvalues of a
stochastic matrix $A$, this matrix may not be any solution \eqref{problem:old}.
Even worse, it is not clear if there exists any stochastic matrix such that the functional becomes zero.
Therefore, we propose to consider an extended problem which is based on the following considerations.
Let $\mu^{[2]} \in \R^{2,2}$ with eigenvalues $\lambda = \lambda_R \pm i \lambda_I$, $\lambda_I \not = 0$.
Then there exists an invertible matrix $T \in \mathrm{SL}(2)$ of determinant 1 such that
\begin{align}
\label{2x2Similarity}
\mu^{[2]} = T \lambda^{[2]} T^{-1}, \qquad \lambda^{[2]} = \left( 
\begin{array}{rr} 
\lambda_R&\lambda_I\\ 
-\lambda_I& \lambda_R
\end{array} 
\right). 
\end{align}
Moreover, it follows from the $QR$ decomposition of matrices that such a matrix $T$ can be uniquely decomposed as 
$$
T = Q \, T_{\alpha,\beta} ,\qquad T_{\alpha,\beta} =
\left( 
\begin{array}{rr} 
\alpha&\beta\\ 
0&\frac{1}{\alpha}
\end{array} 
\right), \qquad \alpha \in \mathbb R_{>0}, \ \ \beta \in \mathbb R, \ \ Q \in \Orth(2),
$$
cf. \cite{Ro06}. 
We obtain 
\begin{equation}\label{decomp}
\mu^{[2]} = Q \, T_{\alpha,\beta} \, \lambda^{[2]} \, T_{\alpha,\beta}^{-1} \, Q^\tT, \qquad 
T_{\alpha,\beta}^{-1} = 
\left(
\begin{array}{rr} 
\frac{1}{\alpha}&-\beta\\ 
0&\alpha
\end{array} 
\right) = T_{\frac{1}{\alpha},-\beta}.
\end{equation}
As a consequence we have by Theorem \ref{thm:Schur}ii) that every (stochastic) matrix 
with eigenvalues in $\Lambda$ can be written in the form
\begin{equation} \label{geht}
Q \, T_{ a b} \, \left( D(\Lambda) + V\right) \, T_{ a  b}^{-1} \, Q^\tT
\end{equation} 
with $Q \in \Orth(n)$ and a blockdiagonal matrix 
$$T_{ a b} = \mathrm{blockdiag} (I_s,T_{a_1b_1} ,\ldots,T_{a_tb_t}),$$
where
$$
a \coloneqq (a_1,\ldots,a_t) \in \mathbb R_{>0}^t, \quad  b \coloneqq (b_1,\ldots,b_t) \in \mathbb R^t.
$$

\begin{example}\label{ex:33_continued}
We continue Example \ref{ex:33}. We have seen that the matrix $A$ in \eqref{Counterexample3x3} 
does not have a decomposition \eqref{isospect}, but can of course be decomposed as in \eqref{geht} with
$a\approx 0.81636$, $b=0$ and
\begin{align} \label{Counterexample3x3_cont}
Q \approx \left(\begin{array}{ccc}
0.57735 & 0.78868 & 0.21132 \\
0.57735 & -0.57735 & 0.57735 \\
0.57735 & -0.21132 & 0.78868
\end{array}\right), \quad
V \approx \left(\begin{array}{ccc}
1 & 0.42152 & 0.42834 \\
0 & 0 & 0 \\
0 & 0 & 0
\end{array}\right).
\end{align}
\end{example}

Instead of problem \eqref{problem:old}, we consider
$F: \R^{n,n}\times\R^{n,n}\times\R^{n,n}\times\R_{>0}^t\times\R^t \rightarrow \mathbb R$ defined by 
\begin{align}
\label{functional:new}
F(S,Q,V, a, b) &\coloneqq \frac{1}{2} \| S \circ S - Q \, T_{ a  b} \, \left( D(\Lambda)+V \right) \, T^{-1}_{ a  b} \, Q^\tT\|_F^2
\end{align}
and propose to solve
\begin{align} \label{model:new}
\min_{S,Q,V,a,b} F(S,Q,V,a,b) \quad \mathrm{ subject \; to} \quad  (S,Q,V, a, b) \in\mathcal{M}_t,
\end{align}
where 
$$
\mathcal{M}_t \coloneqq \mathcal{OB}(n) \times \Orth(n) \times \mathcal V_t \times \mathbb R^t_{>0} \times \mathbb R^t.
$$
If the entries of $\Lambda$ are the eigenvalues of a stochastic matrix, then the minimum of $F$ is zero.

In the following proposition, bounded sets on the manifold $\mathcal{M}_t$ address boundedness w.r.t.~the geodesic distance 
as introduced in Section \ref{sec:specmani}.

\begin{proposition}\label{lem:ExistenceMinimizers}
The function $F\colon 
\R^{n,n}\times\R^{n,n}\times\R^{n,n}\times\R_{>0}^t\times\R^t 
\rightarrow \mathbb R$ in \eqref{model:new} is lower level bounded, i.e.,
$$
\mathrm{lev}_c F \coloneqq \{ x = (S,Q,V, a, b) : F(x) \le c\}
$$
are bounded for all $c \in \mathbb R$, in particular the components of $a \in \R_{>0}^t$ are bounded away from zero.
Then $\mathrm{inf} \, F$ is attained and the set of minimizers is compact.
\end{proposition}

\begin{proof}
By orthogonality of $Q$ we have
$$
F(x) 
= \frac{1}{2} \| Q^\tT (S \circ S) Q -  T_{ a  b} \, \left( D(\Lambda)+V \right) \, T^{-1}_{ a  b} \|_F^2
= \frac{1}{2} \|Q^\tT (S \circ S) Q - (\tilde D(\Lambda,a,b) + \tilde V) \|_F^2,
$$
where $\tilde V \in \mathcal V_t$ and
$$
\tilde D(\Lambda,a,b) = \mathrm{blockdiag} (\lambda_1,\ldots,\lambda_s, B_1,\ldots,B_t),
$$
$$
B_k \coloneqq 
\left(
\begin{array}{rr}
\lambda_{s+k,R} - \frac{b_k}{a_k} \lambda_{s+k,I}& (a_k^2 + b_k^2) \lambda_{s+k,I}\\
-\frac{1}{a_k^2} \lambda_{s+k,I} & \lambda_{s+k,R} + \frac{b_k}{a_k} \lambda_{s+k,I}
\end{array}
\right), \qquad k=1,\ldots,t.
$$
Since $\mathcal{OB}(n) \times \Orth(n)$ is compact, we know that 
$Q^\tT (S \circ S) Q$ has bounded entries.
Therefore $F(X) \le c$ implies that the entries of $\tilde V$ as well as those of $B_k$, $k=1,\ldots,t,$ are bounded.
Since $\lambda_{s+k,I} \not = 0$,  we obtain that $a_k^2 + b_k^2$ and $\frac{1}{a_k^2}$ are bounded, which implies that 
$|b_k|$ is bounded and $a_k$ is bounded from above and away from zero.
In summary, the level sets of $F$ are bounded.
Since $F$ is continuous, the rest of the claim follows by standard arguments from optimization, see, e.g., \cite{BNO03}.
\end{proof}

\section{Geometric Conjugate Gradient Algorithms} \label{sec:algs}
%
We want to find a minimizer of $F\colon \mathcal M_t \rightarrow \mathbb R$ by 
applying a CG algorithm.

\paragraph{CG algorithm on linear spaces.}
To this end, we start with briefly recalling the CG-algorithm to solve a linear system of equations  $Ax =b$
with a symmetric, positive definite matrix $A$. The solution can be found by iteratively
minimizing the functional
\begin{equation} \label{simple}
f(x) = \frac12 x^\tT A x - b^\tT x.
\end{equation} 
The CG method is given in Algorithm \ref{alg:ConjugateGradientMethodUpdates}.

\begin{algorithm}[ht]    
\caption{CG algorithm for linear systems $Ax = b$} \label{alg:ConjugateGradientMethodUpdates}
\begin{algorithmic}   
    \State  \textbf{Input: }  $x^{(0)} \in\R^n$ 
    \State  \textbf{Initialization: } $d^{(0)} = -g^{(0)} = -\nabla f(x^{(0)}) = b - Ax^{(0)}$
    \For{$k = 0,\ldots$ until a stopping criterion is reached, }
    \vspace{0.5ex}
	\State $\hspace{2ex} \alpha^{(k)}  =  \frac{\bigl\langle -g^{(k)},d^{(k)} \bigr\rangle}{\bigl\langle A d^{(k)}, d^{(k)}\bigr\rangle} $
	\vspace{0.5ex}
 	\State $x^{(k+1)} = x^{(k)}+\alpha^{(k)} d^{(k)}$
 	\vspace{0.5ex}
	\State $g^{(k+1)} = \nabla f(x^{(k+1)}) = Ax^{(k+1)}-b$
	\vspace{0.5ex}
 	\State $\phantom{^{+1}}\beta^{(k)} = \frac{\bigl\langle g^{(k+1)},A d^{(k)}\bigr\rangle}{\bigl\langle d^{(k)}, Ad^{(k)} \bigr\rangle}$
 	\vspace{0.5ex}
 	\State $d^{(k+1)} = \ -g^{(k+1)} + \beta^{(k)} d^{(k)}$
 \EndFor
    \end{algorithmic}
\end{algorithm}

For the quadratic functional \eqref{simple}, 
the Hessian is given by $\nabla^2 f(x^{(k)}) = A$.
Further, it can be shown that $\beta^{(k)}$ can be rewritten as
\begin{align} \label{choice_1}
\beta^{(k)} 
&= \frac{\bigl\langle g^{(k+1)}, g^{(k+1)}\bigr\rangle}{\bigl\langle  g^{(k)}, g^{(k)} \bigr\rangle} \\
&=  \frac{\bigl\langle g^{(k+1)}, g^{(k+1)} - g^{(k)}\bigr\rangle}{\bigl\langle  g^{(k)}, g^{(k)} \bigr\rangle} \label{choice_2}.
\end{align}

The conjugate gradient algorithms for minimizing an arbitrary two times differentiable function 
$F\colon\mathbb R^n  \rightarrow \mathbb R$ differ by the choice of 
$\beta^{(k)}$ and $\alpha^{(k)}$ in Algorithm 
\ref{alg:NonlinearConjugateGradientMethodUpdates}.
Note that the values \eqref{choice_1} and \eqref{choice_2} differ in the general case.
For a good overview over different nonlinear CG algorithms see, e.g., \cite{HZ2005}.

\begin{algorithm}[ht]    
\caption{Generic CG algorithm for two times differentiable functionals 
$F:\R^n\rightarrow\R$} \label{alg:NonlinearConjugateGradientMethodUpdates}
\begin{algorithmic}   
    \State  \textbf{Input: }  $x^{(0)} \in\R^n$ 
    \State  \textbf{Initialization: } $d^{(0)} = -g^{(0)} = -\nabla F(x^{(0)})$
    \For{$k = 0,\ldots$ until a stopping criterion is reached, }
	\State Choose a step size $\alpha^{(k)}$
 	\State $x^{(k+1)} = x^{(k)}+\alpha^{(k)} d^{(k)}$
	\State $g^{(k+1)} = \nabla F(x^{(k+1)}) $
 	\State $ d^{(k+1)} = - g^{(k+1)} + \beta^{(k)} d^{(k)} - \theta^{(k)} \left( g^{(k+1)} - g^{(k)} \right)$ for certain coefficients $\beta^{(k)}$, $\theta^{(k)}$
 \EndFor
    \end{algorithmic}
\end{algorithm}

For $\theta^{(k)} = 0$ and
\begin{align*}
\beta^{(k)} = \frac{\bigl\langle \nabla^2 F(x^{(k+1)})d^{(k)},g^{(k+1)}\bigr\rangle}{\bigl\langle \nabla^2 F(x^{(k+1)}d^{(k)},d^{(k)}\bigr\rangle} ,
\end{align*}
the resulting algorithm was proposed by Daniel \cite{Da67}.
Using $\theta^{(k)} = 0$ and $\beta^{(k)}$ as in \eqref{choice_1}, we obtain the  Fletcher-Reeves algorithm 
and $\beta^{(k)}$ as in \eqref{choice_2}  the 
Polak-Ribière-Polyak  algorithm, cf. \cite{ZZL06}. 
Finally, the Polak-Ribière-Polyak algorithm can be modified by setting
\begin{equation} \label{Modified PRP}
\theta ^{(k)} = \frac{\bigl\langle g^{(k+1)}, d^{(k)} \bigr \rangle}{\| g^{(k)}\|^2}.
\end{equation}
This modified Polak-Ribière-Polyak  algorithm will be our method of choice. 
It has the following properties whose brief proof we give for convenience, cf. \cite[Theorem 2.1]{ZZL06}.

\begin{proposition}\label{PropertiesMPRP}
For Algorithm \textup{\ref{alg:NonlinearConjugateGradientMethodUpdates}} 
with update \textup{\eqref{choice_2}} and \textup{\eqref{Modified PRP}} the 
following holds true:
\begin{itemize}
\item[\textup{i)}] 
If the step sizes $\alpha^{(k)}$ are found by exact line search, i.e., 
$\alpha^{(k+1)} = \mathrm{argmin}_{\alpha} F(x^{(k)} \\+ \alpha d^{(k)})$,
then $\theta^{(k)}=0$ for all $k\in\N$, so that we obtain the usual Polak-Ribière-Polyak Algorithm.
\item[\textup{ii)}] Independently from the line search, $d^{(k)}$ is a descent 
direction at $x^{(k)}$ for all $k\in\N$.
\end{itemize}
\end{proposition}

\begin{proof}
i) 
Since the line search is exact we have for
the function $\alpha \mapsto F(x^{(k)}+\alpha d^{(k)})$ that
\begin{align*}
0 =  \bigl\langle \nabla F(x^{(k)}+\alpha^{(k)} d^{(k)}),d^{(k)} \bigr\rangle = \bigl\langle \nabla F(x^{(k+1)}),d^{(k)}\bigr\rangle 
= \bigl\langle g^{(k+1)},d^{(k)}\bigr\rangle.
\end{align*}
ii){ For $k=0$ it holds $d^{(0)}=-g^{(0)}$. For $k\geq 1$ we set $y^{(k)} 
\coloneqq g^{(k+1)}-g^{(k)}$ and obtain
\begin{align*}
\bigl\langle g^{(k+1)},d^{(k+1)} \bigr\rangle 
\vspace{1ex}
&= -\|{g^{(k+1)}}\|^2 + \beta^{(k)} \bigl\langle g^{(k+1)},d^{(k)}\bigr\rangle - \theta^{(k)}\bigl\langle g^{(k+1)},y^{(k)}\bigr\rangle  \\
&= -\|g^{(k+1)}\|^2 + \frac{\bigl\langle g^{(k+1)},y^{(k)}\bigr\rangle \bigl\langle g^{(k+1)},d^{(k)}\bigr\rangle}{\|{g^{(k)}}\|^2} 
- \frac{\bigl\langle g^{k+1},d^{(k)}\bigr\rangle \bigl\langle g^{(k+1)},y^{(k)}\bigr\rangle}{\|g^{(k)}\|^2} \\
\vspace{1ex}
&= - \|g^{(k+1)}\|^2 <0. 
\end{align*}}
\end{proof}

\paragraph{CG algorithms on manifolds.}
Let $\mathcal M$ be a complete, connected $d$-dimensional Riemannian manifold with tangent space $T_x \mathcal M$ in $x \in \mathcal M$ and
Riemannian metric $\| \cdot \|_x$. By  $T\mathcal{M}$ we denote the tangent bundle of $\mathcal M$.
The geometric version of the CG algorithm replaces the gradient/Hessian 
by the Riemannian gradient/Hessian $\nabla_{\mathcal M}$/$\nabla_{\mathcal M}^2$ on the manifold,
and incorporate a vector transport $\mathcal{T}$ 
between tangent spaces of the manifold in order to make the addition of the vectors $d^{(k)}$, which are in different tangent spaces, possible. 

Recall that a map $\mathcal{R}\colon T\mathcal{M} \rightarrow \mathcal{M}$ is 
called 
a \emph{retraction} (on $\mathcal{M}$), if 
\begin{itemize}
\item[i)] $\mathcal{R}_x(0_x)=x$, where $0_x$ denotes the zero vector in $T_x\mathcal{M}$, and
\item[ii)] $D\mathcal{R}_x(0_x) = \text{id}_{T_x\mathcal{M}}$ with the canonical identification $T_0(T_x\mathcal{M}) = T_x\mathcal{M}$.
\end{itemize}
In particular, on a complete manifold, the \emph{exponential map} $\exp_x\colon T_x\mathcal{M} \rightarrow \mathcal M$ is a retraction.
Fixing a retraction we can define a vector transport between tangent spaces.\\
A \emph{vector transport} $\mathcal{T}: T\mathcal{M}\times T\mathcal{M}\rightarrow T\mathcal{M}$ 
associated to a retraction $\mathcal{R}\colon T\mathcal{M}\rightarrow \mathcal{M}$ is a smooth mapping defined by
\begin{align*}
\left( (x,\theta),(x,\xi) \right) \mapsto (\mathcal{R}_x(\theta),\mathcal{T}_{x,\theta} \xi)
\end{align*}
with the following properties
\begin{itemize}
\item[i)](Consistency)  $\mathcal T_{x,0_x}\xi =\xi$ for all $\xi\in T_x\mathcal{M}$.
\item[ii)](Linearity)   $\mathcal T_{x,\theta}(a\xi+b\zeta) = a \mathcal T_{x,\theta}\xi
                          + b \mathcal T_{x,\theta}\zeta$ for all $\theta,\xi,\zeta\in T_x\mathcal{M}$, $a,b\in\R$.
\end{itemize}
In particular, the \emph{parallel transport} $\mathcal{P}_{x,\theta} (\xi)$ 
along a geodesic starting at $x$ with tangent vector $\theta$ is a vector transport.
A vector transport different from the parallel transport yields often computationally less expensive algorithms with similar convergence properties.

Now the modified Polak-Ribière-Polyak CG-algorithm on a manifold $\mathcal M$ 
is given by 
Algorithm \ref{alg:cg_1}. 

\begin{algorithm}    \caption{\textbf{Geometric Modified Polak-Ribière-Polyak Algorithm} (GMPRP)} \label{alg:cg_1}
\begin{algorithmic}
   \State \textbf{Input}: initial point $x^{(0)} \in \mathcal{M}$       
   \State \textbf{Initialization}:  $g^{(0)} = \nabla_{\mathcal M} F\bigl(x^{(0)}\bigr)$, 
		$d^{(0)}= - g^{(0)}$, $k=0$     
   \While{$\Vert g^{(k)} \Vert_{x^{(k)}}>0$ }  \vspace{0.5ex}
   \State Find a step size $\alpha^{(k)}$ \vspace{0.5ex}
   \State  $x^{(k+1)} = \mathcal{R}_{x^{(k)}}(\alpha^{(k)} d^{(k)})$
   \vspace{0.5ex}
   \State $g^{(k+1)} = \nabla_{\mathcal{M}} F(x^{(k+1)})$
   \vspace{0.5ex}
   \State $\hphantom{^{+1}} y^{(k)}  = g^{(k+1)} - \mathcal{T}_{x^{(k)}, \alpha^{(k)} d^{(k)} } g^{(k)}$
   \vspace{0.5ex}
   \State $\hphantom{^{+1}}\tilde{d}^{(k)} = \mathcal{T}_{x^{(k)}, \alpha^{(k)}d^{(k)}} d^{(k)}$
   \vspace{0.5ex}
   \State $\hphantom{^{+1}}\beta^{(k)} = \frac{\langle g^{(k+1)},y^{(k)} \rangle_{x^{(k+1)}} }{ \|g^{(k)}\|_{x^{(k)}}^2 } $
   \vspace{0.5ex}
   \State $\hphantom{^{+1}}\theta^{(k)} = \frac{\langle g^{(k+1)}, \tilde{d}^{(k)}  \rangle_{x^{(k+1)}}}{\| g^{(k)}\|_{x^{(k)}}^2}$
   \vspace{0.5ex}
   \State $d^{(k+1)} = -g^{(k+1)} + \beta^{(k)} \tilde{d}^{(k)} - \theta^{(k)} y^{(k)}$
   \vspace{0.5ex}
   \State $k = k+1$
\EndWhile
\end{algorithmic}
\end{algorithm}

The geometric GMPRP algorithm has analogous properties as in the Euclidean setting.
\newline 

\begin{proposition} \label{thm:AlwaysDescentDirection}
For Algorithm \textup{\ref{alg:cg_1}} the following holds true:
\begin{itemize}
\item[\textup{i)}] If the step sizes $\alpha^{(k)}$ are generated by an exact 
line search and parallel transport is chosen as vector transport, 
then it holds $\theta^{(k)}=0$ for all $k\in\N$. 
\item[\textup{ii)}] Independently from the line search and the vector 
transport, $d^{(k)}$ is a descent direction at $x^{(k)}$ for all $k\in\N$.
\end{itemize}
\end{proposition}

\begin{proof}
i)
In case of exact line search it holds
\begin{align*}
\bigl\langle g^{(k+1)}, \tilde{d}^{(k)} \bigr\rangle_{x^{(k+1)}} 
= \bigl\langle \nabla_\mathcal{M} F(x^{(k+1)}), \mathcal{T}_{x^{(k)},\alpha^{(k)} d^{(k)}} d^{(k)} \bigr\rangle_{x^{(k+1)}} = 0,
\end{align*}
cf. \cite[p.131]{Sm94}. 

ii) The assertions follow as in the proof of Theorem \ref{PropertiesMPRP} by straightforward computation.
\end{proof}

In the numerical part, we will determine $\alpha^{(k)}$ by the Line Search Algorithm \ref{alg:AlphasqrLineSearch}.
This search was proposed in \cite{ZJB2016}, but with another initial step size given by
\begin{align} \label{eqn:GaussNewtonLikeStepsize}
\tilde{\alpha}^{(0)} = \frac{\langle d, \nabla_{\mathcal{M}} F(X) \rangle_X }{\norm{Df(X)[d]}_F^2}, \ \ \ X\in\mathcal{M}, \ d\in T_X\mathcal{M}
\end{align}
where $f$ is the matrix function with $F(X) = \frac{1}{2}\|f(X)\|_F^2$ and 
$Df(X)[d]$ denotes the usual derivative in an euclidean space. 
In our numerical experiments we have observed that the initial Newton step size in Algorithm \ref{alg:AlphasqrLineSearch}
with the approximated Hessian leads to less updates during line search than 
with initialization \eqref{eqn:GaussNewtonLikeStepsize} 
and that this even compensates the fact that the Newton step size is more 
expensive to compute.
We have approximated the Hessian by
\begin{align}
\label{eqn:ApprHessian}
\nabla_\mathcal{M}^2 F(x)d 
\ \approx \ \norm{d} \frac{ \nabla_\mathcal{M} F\bigl( \gamma_{x,d}\bigl(\frac{h}{\norm{d}}\bigr) \bigr) - \nabla_\mathcal{M} F(x) }{h}
\end{align}
in our numerical examples.  
Note that, compared to the line search in \cite{ZJB2016}, we have inserted an additional step which is required for \eqref{eqn:1} in the convergence proof. In the proof of \cite[Theorem 3.4]{ZJB2016} a constant initial stepsize $\tilde{\alpha}^{(0)}$ or a similar step 
to our additional one is needed for the same reason. In Section 
\ref{sec:numerics} we compare the behaviour of both line search algorithms. 
Alternatively, we could use the classical Armijo Line Search with the same initial $\tilde{\alpha}^{(0)}$. 
Indeed, we have also implemented this step size search algorithm, but it does 
not appear cheaper than the first one.
\newline

\begin{algorithm}[ht]
	\caption{ (\textbf{Line Search Algorithm}) } 	\label{alg:AlphasqrLineSearch}
	\begin{algorithmic}
		\State \textbf{Parameters:}  $0 < \tau < 1$, $0 < \delta < 1$, 
		alternative step size $\alpha^\star$
		\State \textbf{Input:} smooth function $F\colon{\mathcal M} \to [0,\infty)$, start point $x \in {\mathcal M}$, 
		descent direction $d \in \mathrm T_{x} {\mathcal M}$
		\State \textbf{Output:} step size $\alpha^{(k)} = \tilde{\alpha}^{(l)}$
		\State \textbf{Initialization:}
		$l = 0$,
		\[\tilde{\alpha}^{(0)} =
		\begin{cases}
		\left| \frac{ \bigl\langle d, \nabla_{{\mathcal M}} F(x) \bigr\rangle_x}{ \bigl\langle d, \nabla^2_{\mathcal M} F(x) d\bigr\rangle_x } \right|,
		& \bigl\langle d, \nabla^2_{{\mathcal M}}F(x) d \bigr\rangle_x \ne 0,\\
		\alpha^\star, & \text{ else}
		\end{cases}
		\]
		\While{$F \left( \mathcal{R}_x ({\tilde{\alpha}}^{(l)} d ) \right) - 
		F(x) \ge - \delta {{}\tilde{\alpha}^{(l)}}^{2}\norm{d}_x^2$}
		\State $\tilde{\alpha}^{(l+1)} = \tau \tilde{\alpha}^{(l)}$
		\State $l = l+1$
		\EndWhile
		\State \textbf{Additional Step:}
		\If{$l=0$}
		\While{$F \left( \mathcal{R}_x (\tilde{\alpha}^{(l)} d ) \right) - F(x) 
		< - \delta {{}\tilde{\alpha}^{(l)}}^{2} \norm{d}_x^2$}
		\State $\tilde{\alpha}^{(l+1)} = \tau^{-1} \tilde{\alpha}^{(l)}$
		\State $l= l+1$
		\EndWhile
		\State $l = l-1$
		\EndIf
	\end{algorithmic}	
\end{algorithm}

\section{Special Manifolds}\label{sec:specmani}
In this section, we provide the quantities required in the CG algorithms for the special manifolds appearing in
$$
\mathcal{M}_t \coloneqq \mathcal{OB}(n) \times \Orth(n) \times \mathcal V_t \times \mathbb R^t_{>0} \times \mathbb R^t.
$$
The spaces $\mathcal{V}_t$ and $\mathbb R^t$ are linear spaces,
$\OB$ and $\Orth(n)$ are smooth compact matrix manifolds 
of dimension $n(n-1)$ with tangent spaces 
\begin{align*}
T_S{\OB} &= \{ \Xi \in \mathbb{R}^{n,n}: \diag(S\, \Xi^T)=0 \},\\
T_Q{\Orth(n)} &= Q \, \text{Skew}(n), 
\end{align*}
where $\text{Skew}(n)$ denotes the linear space of $n \times n$ skew-symmetric matrices. 
Note that $\Orth(n)$ is actually not a connected manifold, but can be split into the two connected components $\mathrm{{SO}(n)}$ 
and $-\mathrm{{SO}(n)}$ and since the functionals \eqref{problem:old} and \eqref{functional:new} are invariant under sign changes for $Q$, this fact is not of relevance for our algorithms.
The manifolds $\OB$ and $\Orth(n)$ are isometrically  embedded into $\mathbb R^{n,n}$ and the tangent spaces are equipped with the inner product
\begin{align*}
\langle \Xi_1,\Xi_2 \rangle_S &= \tr (\Xi_1^\tT \Xi_2), \quad \Xi_1,\Xi_2\in T_S\OB, \\
\langle \Xi_1,\Xi_2 \rangle_Q &= \tr (\Xi_1^\tT \Xi_2), \quad \Xi_1,\Xi_2 \in T_Q \Orth(n).
\end{align*}
In particular, these inner products are independent of $S$, resp. $Q$.
The tangent space of the positive numbers at $a \in  \mathbb R^t_{>0}$ is $T_{a} \mathbb R^t_{>0} = \mathbb R^t$  with inner product
\begin{equation} \label{mult}
\bigl\langle \xi^{(1)},\xi^{(2)} \bigr\rangle_{a} = \frac{\xi^{(1)}\xi^{(2)}}{a^2}, \qquad \xi^{(1)},\xi^{(2)} \in \mathbb R^t,
\end{equation}
where the product (and quotients) of vectors is meant componentwise.
Finally, we have for 
$x=(S,Q,V,a,b) \in \mathcal M_t$ 
that 
$T_x{\mathcal M_t} = T_S{\OB} \times T_Q{\Orth(n)} \times \mathcal V_t \times \mathbb R^t \times \mathbb R^t$
with the Riemannian metric \eqref{mult} on the fourth space and the Euclidean 
Riemannian metric on the other spaces (independent of the concrete $S,Q,V,b$).
The Riemannian gradient of $F$ requires
the orthogonal projections of $\Xi \in \mathbb R^{n,n}$ onto the tangent spaces which
are given by
\begin{align*}
\Pi_{T_S\OB}(\Xi) &= \Xi - \diag(S \, \Xi^\tT)S,  \\
\Pi_{T_Q\Orth(n)}(\Xi) &= Q \, \frac12 (Q^\tT \Xi - \Xi^\tT Q).  \\ 
\end{align*}
Then we have the following lemma.

\begin{lemma}\label{gradients}
The Riemannian gradient of $F = F(S,Q,V,a,b)$ in \eqref{model:new} reads as
$\nabla_{{\mathcal M}_t} F = (\nabla_S F,\nabla_Q F,\nabla_V F, \nabla_{a} F,\nabla_{b} F )$, where
    \begin{align}
   \nabla_S F     &= \Pi_{T_S \OB} \left( 2 S \circ H \right),  \\
    \nabla_Q F     &= \Pi_{T_Q \O (n)} \left( -(H^\tT G + H  G^\tT) \, Q \right), \\
    \nabla_V F     &= \Pi_{T_V \mathcal{V}_t } \left(-T_{ab}^\tT \, Q^\tT H  Q 
    T_{ab}^{-\tT} \right),\\
    \nabla_{a_k} F &= \Bigl\langle Q^\tT \, (G^\tT  H - H G^\tT) \, Q \, T_{ab}^{-\tT} , \mathrm{blockdiag} \bigl( 0_{s+ 2(k-1)},M_{a_k} ,0_{2(t-k-1)} \bigr) \Bigr\rangle,\\
    \nabla_{b_k} F &= \Bigl\langle Q^\tT \, (G^\tT  H - H G^\tT) \, Q \, T_{ab}^{-\tT} , \mathrm{blockdiag} \bigl( 0_{s+ 2(k-1)},N,       0_{2(t-k-1)} \bigr) \Bigr\rangle,
       \end{align}
$k=1,\ldots,t$, where $G \coloneqq Q \, T_{ a  b}\,  \left( D(\Lambda)+V \right) \, T^{-1}_{ a  b} \, Q^\tT$, $H  \coloneqq S \circ S - G$, and
$$
N \coloneqq \begin{pmatrix} 0&1\\0&0\end{pmatrix}, \quad 
M_\alpha \coloneqq \begin{pmatrix} \alpha^2&0\\0&-1\end{pmatrix}.
$$
\end{lemma}

\begin{proof}
The computation of the first three Riemannian gradients 
is based on the fact that the Riemannian gradients are just orthogonal projections of the Euclidean gradients onto the
respective tangent spaces and the Euclidean matrix gradients follow by straightforward computations.

The fourth and fifth Riemannian gradients can be directly obtained from the chain rule for computing matrix gradients,
the fact that the mappings 
$f_1: \mathbb R \to \SL(2)$ and $f_2: \mathbb R_{>0} \to \SL(2)$
defined by
$$
f_1(\beta) = \begin{pmatrix} \alpha&\beta\\0&\frac{1}{\alpha} \end{pmatrix} \quad \mathrm{and} \quad f_2(\alpha) 
=  \begin{pmatrix} \alpha&\beta\\0&\frac{1}{\alpha} \end{pmatrix}
$$
have the differentials
$$
Df_1(\beta)   = N \quad \mathrm{and} \quad Df_2(\alpha) = \frac{1}{\alpha^2}M_\alpha,$$
respectively, 
and by regarding \eqref{mult} in the computation of the fourth gradient.
\end{proof}

For implementation, note that we can rewrite $\nabla_{a_k} F$ and $\nabla_{b_k} F$ as 
\begin{align*}
\nabla_{a_k} F &= a_k^2 A_{s+2k-1,s+2k-1} - A_{s+2k,s+2k}, \\
\nabla_{b_k} F &= A_{s+2k-1,s+2k}
\end{align*}
with $A \coloneqq Q^\tT \, (G^\tT  H - H G^\tT) \, Q \, T_{ab}^{-\tT}$. 
\newline 

Next we give retractions and corresponding transport maps for the involved manifolds.
Since $\mathcal{V}_t$ and $\mathbb R^t$ are linear spaces, we  have 
$$
\mathcal{R}_V(\Xi) = V+\Xi, \quad \mathcal{R}_b{\xi} = b + \xi.
$$
Retractions on $\OB$ follow from the retractions on the $(n-1)$-sphere
\begin{align*}
   \mathcal{R}_s(\xi) &= \frac{s+\xi}{\norm{s+\xi}}, \qquad \qquad \qquad \xi 
   \in T_{s} \mathbb S^{n-1}\\
    \exp_s \xi &= \cos(\| \xi \|)s +   \frac{\sin(\|\xi\|)}{\|\xi\|}\xi
\end{align*}
and are given by
\begin{align*}
\mathcal{R}_S(\Xi) &= 
\Bigl(\text{diag}\bigl((S+\Xi)(S+\Xi)^T\bigr)^{-\frac{1}{2}}\Bigr)(S+\Xi), \qquad \Xi \in T_S\OB,\\
\exp_S (\Xi) &= 
  \cos \Bigl( \text{diag} \bigl(\Xi \Xi^\tT\bigr)^{\frac{1}{2}} \Bigr) S 
+ \sin \Bigl( \text{diag} \bigl(\Xi \Xi^\tT\bigr)^{\frac{1}{2}} \Bigr) \, \text{diag} (\Xi \Xi^\tT)^{-\frac{1}{2}} \, \Xi,
\end{align*}
which are both smooth mappings.
For retractions on $\Orth(n)$, recall that the sum of the identity matrix and any skew-symmetric matrix is invertible and that for an invertible matrix $A\in \mathbb{R}^{n,n}$, 
there is a unique matrix $Q\in \Orth(n)$ and an upper diagonal matrix with positive diagonal entries $R\in \mathbb{R}^{n,n}$ 
such that $A=QR$. We denote the corresponding matrix $Q$ as $\mathrm{qf}(A)$.  
Well-known retractions on $\Orth(n)$ are given by
\begin{align}      
   \mathcal{R}_Q( \Xi)    &=     Q \, \mathrm{qf}(I+Q^\tT\, \Xi) = \mathrm{qf}(Q+\Xi), \qquad \Xi \in T_Q \Orth(n),\label{eqn:RetractionsReal2}\\
 \exp_{Q}( \Xi) &= Q \, \Exp(Q^\tT\, \Xi),
    \end{align}
see \cite[Example 4.1.2]{AMS09}, which are both smooth.
Another retraction which we do not apply here is given via the Cayley transform, see also \cite[Example 4.1.2]{AMS09}. 
\newline
Finally, we use as retraction on $\mathbb R_{>0}$ the exponential map
$$
\mathcal{R}_{a_k}(\xi) = \exp_{a_k} \xi = {a_k} \mathrm{e}^\frac{\xi}{{a_k}}, \qquad \xi \in \mathbb R,
$$
cf. \cite[Theorem 2.14]{ALRL11}. By straightforward computation one verifies that the geodesic distance on $\R_{>0}^t$ with inner product \eqref{mult} is given by
\begin{align}
\label{eqn:GeoDistPos}
\text{dist}_{\R_{>0}^t}(a,\tilde{a}) = \|\log(a)-\log(\tilde{a})\|_2,
\end{align}
where the logarithm is meant componentwise.

For the vector transport we use as in \cite{ZJB2016},
    \begin{align}
    \mathcal{T}_{S,\Theta} \Xi &= \Pi_{T_{\mathcal{R}_S(\Theta)}\OB} (\Xi),  \quad \Theta, \Xi \in T_S\OB,  \label{eqn:ProjTransport} \\  
    \mathcal{T}_{Q,\Theta} \Xi &= \Pi_{T_{\mathcal{R}_Q(\Theta)}\Orth(n)} (\Xi), \quad \; \Theta, \Xi \in T_Q\Orth(n),\\
		\mathcal{T}_{a_k,\theta} \xi &= \xi, \qquad \qquad \qquad \quad \theta, \xi \in T_{a_k} \mathbb R_{>0}.
    \end{align}
		
\begin{remark} \label{rem:parallel}		
Alternatively we could apply the following parallel transport maps:   
The parallel transport of $\Xi \in T_S\OB$ along geodesics in direction $\Theta \in T_S\OB$ at $S\in\OB$ is determined row-wise by the
parallel transport on the \textup{(}$n-1$\textup{)}-sphere, see 
\textup{\cite{HA17}}:
\begin{align*}
\mathcal{P}_{s,\theta}(\xi) 
&= \left(I_n + \bigl( \mathrm{cos} \bigl(\norm{\theta}\bigr) - 1 \bigr) \frac{\theta \theta^\tT}{\norm{\theta}^2} -
\mathrm{sin}\bigl(\norm{\theta}\bigr)\, \frac{s\theta^\tT}{\|\theta\|} \right)\xi, \qquad \xi \in T_s \mathbb S^{n-1}. 
\end{align*}
Further, we have
\begin{align}
\mathcal{P}_{Q,\Theta}(\Xi) &= Q \, \mathrm{Exp}\left(\frac{Q^\tT\Theta}{2}\right)Q^\tT \, \Xi \,  \mathrm{Exp}\left(\frac{Q^\tT\Theta}{2}\right), \qquad \Theta, \Xi \in T_Q \Orth(n),  \\
\mathcal{P}_{{a_k},\theta}(\xi) &=  e^\frac{\theta}{{a_k}}\xi = \frac{\mathcal{R}_{a_k}(\theta) }{{a_k}} \xi, 
\qquad \qquad \qquad \qquad \qquad \theta, \xi \in T_{a_k} \mathbb R_{>0}. \label{eqn:ParTransport}
\end{align}
In our numerical examples we have not seen advantages of the parallel transport over the vector transports in \eqref{eqn:ProjTransport}. 
\end{remark}

For our special manifold $\mathcal M = \mathcal M_t$ we have the following 
convergence result of Algorithm~\ref{alg:cg_1}.

\begin{thm} \label{thm:SameConvergenceTheorem}
Let $(x^{(k)})_k$ with 
$x^{(k)} = (S^{(k)},Q^{(k)},V^{(k)},a^{(k)},b^{(k)}) \in \mathcal M_t$ 
be the sequence generated by Algorithm~\textup{\ref{alg:cg_1}} and the Line 
Search Algorithm~\textup{\ref{alg:AlphasqrLineSearch}}. 
If the algorithm does not stop with $\Vert \nabla_{\mathcal{M}_t} F(x^{(k_0)})\Vert_{x^{(k_0)}} = 0$ in some step $k_0$, then it holds
\begin{align*}
\liminf_{k\to\infty} \bigr \Vert \nabla_{\mathcal{M}_t} F\bigl(x^{(k)}\bigr)\bigr \Vert_{x^{(k)}} = 0.
\end{align*}
\end{thm}

The proof was given for the manifold $\widetilde{\mathcal M}_t$ and the 
function $\tilde F$ in \cite{ZJB2016}.
The assertion can be shown for our more general setting in a similar way based on the fact that
the line search algorithm ensures that the sequence $\{F(x^{(k)})\}_k$ is 
monotone decreasing and
that by Proposition \ref{lem:ExistenceMinimizers} the level sets $\mathrm{lev}_{F(x^{(0)})} F$ are bounded
so that all iterates are in this compact set. For convenience we give the proof in the appendix.

\section{Numerical Results} \label{sec:numerics}
%
In this section, we present numerical results using 
Algorithm \ref{alg:cg_1} with vector transport \eqref{eqn:ProjTransport} 
and line search given by Algorithm \ref{alg:AlphasqrLineSearch} for minimizing
\begin{itemize}
\item[(I)] model \eqref{problem:old} from \cite{ZJB2016},
\item[(II)] our new model \eqref{model:new}.
\end{itemize}
The algorithms were implemented in MATLAB R2019b and executed on a computer 
with Intel Core i5-7300U processor and 2 cores at 2,6GHz/ 2712 MHz and 8 
GB RAM. \newline
In the line search of Algorithm \ref{alg:AlphasqrLineSearch}, we set $\delta = 
10^{-4}$ and $\tau = 0.5$ as it was also used in \cite{ZJB2016}. We set our 
initial stepsize to $\alpha^\star = 1.4$ in Method I if $\|d\|<10^{-5}$ in 
\eqref{eqn:ApprHessian} or if $\langle d,\nabla_\mathcal{M}^2 F(x)d \rangle < 
10^{-12}$ and in Method II to $\alpha^\star = 1.6$ if  $\|d\|_x<10^{-5}$ or 
$\langle 
d,\nabla_\mathcal{M}^2 F(x)d \rangle_x < 10^{-10}.$
We use the following initialization.

\paragraph{Initialization.} For the initialization of the descent algorithm we use the MATLAB function \texttt{schur}, 
which computes a decomposition as in Theorem \ref{thm:Schur}(ii), and choose
\begin{align}
\hat{S} &= \texttt{rand}(n,n), \ \ \ S^{(0)} = \left( 
\textup{diag}(\hat{S}1_n)^{-1}\hat{S} \right)^{\frac{1}{2}} \in \OB, \nonumber  
\\
[Q^{(0)},\hat{V}] &= \texttt{schur}(S^{(0)}\circ S^{(0)},\texttt{'real'}), \ \ \ V^{(0)} = \Pi_{\mathcal{V}_t} \hat{V} = W\circ\hat{V},
\label{eqn:StartPoints}
\end{align}
where the square root is meant componentwise and
\begin{align*}
W_{ij} = \left\{\begin{array}{ll} 0, & j\leq i \ \text{or} \ i = j-1, j = s+2k, k\in\{ 1,\ldots,t\},   \\
         1, & \text{otherwise}. 
				\end{array}\right. 
\end{align*}
For model \eqref{model:new} we additionally set 
\begin{align*}
a^{(0)} = 1_t, \ \ \ b^{(0)} = 0_t.
\end{align*}

We compare the eigenvalue sets using Algorithm \ref{alg:ComparisonEigenvalues}.
\begin{algorithm}[ht]    
\caption{Greedy-type computation of a distance of eigenvalue sets} 
\label{alg:ComparisonEigenvalues}
\begin{algorithmic}   
    \State  \textbf{Input: } Eigenvalue sets $\Lambda, \tilde{\Lambda}$ with $n$ elements 
    \State $d(\Lambda,\tilde{\Lambda}) = 0$
 \For{$k = 1,\dots, n$}
    \State  Choose $\lambda'\in\Lambda,\tilde{\lambda}'\in\tilde{\Lambda}$ with $|\lambda'-\tilde{\lambda}'| = \min\{ |\lambda-\tilde{\lambda}|:
    \lambda\in		    \Lambda,\tilde{\lambda}\in\tilde{\Lambda}\}$
	\State $\Lambda = \Lambda\setminus\{\lambda'\}$
	\State $\tilde{\Lambda} = \tilde{\Lambda}\setminus\{\tilde{\lambda}'\}$
	\State $d(\Lambda,\tilde{\Lambda}) =  \max(|\lambda'-\tilde{\lambda}'|,d(\Lambda,\tilde{\Lambda}))$
 \EndFor
    \end{algorithmic}
\end{algorithm}

We consider two kind of numerical examples. 
The first one is similar to that proposed in \cite[Example 5.1]{ZJB2016}.
Here the number $t$ of complex-conjugate eigenvalue pairs increases linearly in the dimension $n$ of the vector space.
In contrast to Method I, our algorithm has additionally to determine an 
increasing number of components of $a$ and $b$ in $T_{ab}$.
Our Method II appears to be slightly slower than the previous one.
In the second example, we fix $n$ and examine how the performance of the algorithm depends on the number $t$ of complex-conjugate eigenvalue pairs.
In the chosen setting, the problem becomes very ill-posed. Interestingly, both 
methods get closer to the prescribed eigenvalues as $t$ grows.

\paragraph{Example with $t \sim n$.}
We start with a similar setting as in \cite[Example 5.1]{ZJB2016}.
The eigenvalues of a stochastic matrix $A \in \mathbb R^{n,n}$ obtained by
\begin{align} 
\tilde A = \texttt{rand}(n,n), \quad A= \left( \textup{diag}(\tilde A 1_n)^{-1} 
\tilde A \right)
\label{eqn:SampleEigenvalues}
\end{align}
with the MATLAB built-in function \texttt{rand} are used as prescribed spectrum $\Lambda$.
For these matrices,  the number $t$ of pairs of complex conjugate eigenvalues is slightly below  $\frac{n}{2}$, 
see Figures \ref{fig:t_against_n},
where we averaged over $10000$ samples.

\begin{figure}[H]
	\centering
	\includegraphics[width=50mm, scale=0.5]{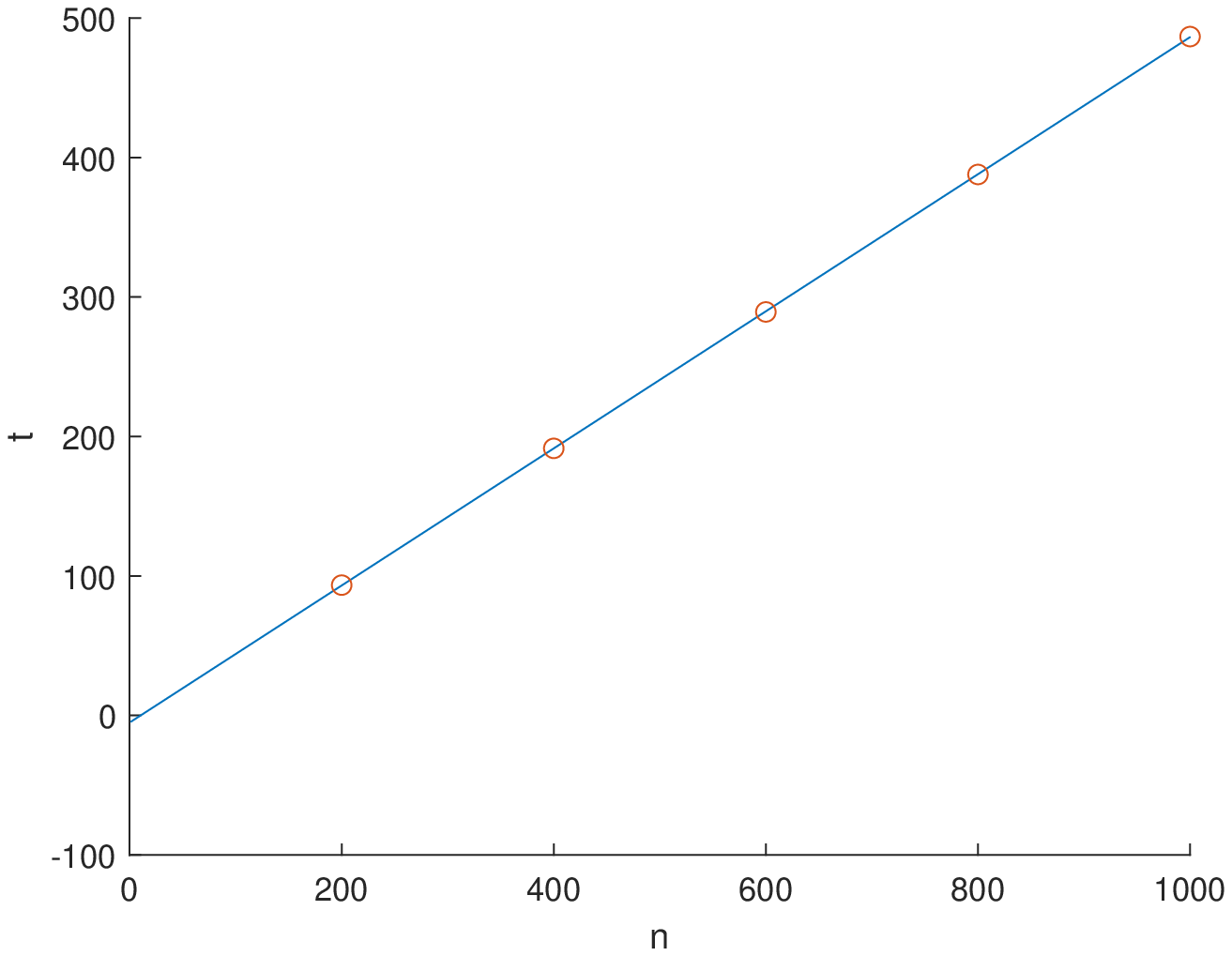}
	\hspace{0.5cm}
	\includegraphics[width=60mm, scale=0.5]{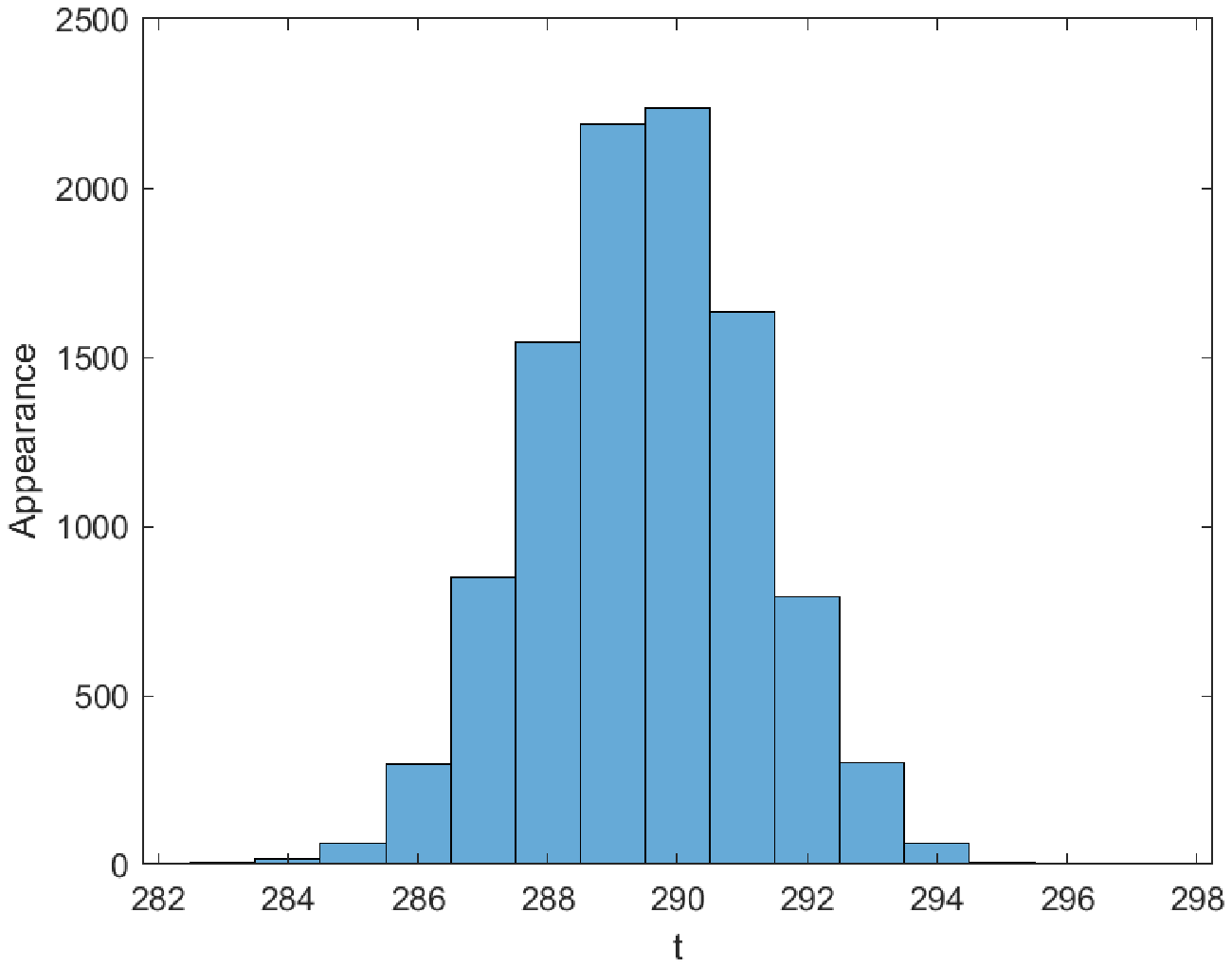}
	\caption{Left: Number $t$ of pairs of complex conjugate eigenvalues of $A$ in \eqref{eqn:SampleEigenvalues} 
	in dependence on the size $n\times n$ of $A$.
	Right: Histogram of $t$ for 10000 samples and $n=600$.}
	\label{fig:t_against_n}
	\end{figure}

We stop the iterations if the following criteria are fulfilled:

\paragraph{Stopping Criterion.}
Method I is stopped if
\begin{align}\label{eqn:StoppingWithFunctionalValueOld}
\norm{S\circ S - Q(L+V)Q^T}_F = \sqrt{2\widetilde{F}(S,Q,V)} < 10^{-12}
\end{align}
and Method II if
\begin{align} \label{eqn:StoppingWithFunctionalValueNew}
\norm{S\circ S - QT_{ab}(L+V)T_{ab}^{-1}Q^T}_F = \sqrt{2 F(S,Q,V,a,b)} < 
10^{-12}.
\end{align}
An alternative would be to stop the algorithms if the eigenvalues of the iterates are close enough to the prescribed eigenvalues 
in the distance measure of Algorithm \ref{alg:ComparisonEigenvalues}. 
Nevertheless, Figure \ref{Fig:test1} indicates that 
the stopping criteria \eqref{eqn:StoppingWithFunctionalValueOld} 
and \eqref{eqn:StoppingWithFunctionalValueNew} yield basically the same eigenvalue distances. 
Besides, our stopping criteria are more convenient, since computation of 
eigenvalues becomes expensive 
if the dimension grows, and since the line search guarantees that the 
functional values decrease monotone, 
which is in general not the case for the eigenvalue distances. 
	
We considered $n=200,400,600,800,1000$ and  compared the number of iterations, 
the runtime and the 
distance of the eigenvalues of the computed stochastic matrix from the given ones,
where we averaged over $50$ samples computed by \eqref{eqn:StartPoints} in each 
dimension. 
Figure \ref{Fig:test1} shows that our new method performs slightly worse, which 
is not surprising since we have to minimize over more variables. 
The exact values of the runtime are shown in Table 
\ref{Tab:runtime_with_add_step}.
Figure \ref{fig:ConvergenceRates} depicts the decay of the target functionals 
in dependence on the number of iterations for $n \in \{200,600,1000\}$.
We also show the corresponding two least square fitting lines.
We observe that Method I and Method II both converge linearly with basically 
the same convergence rate and that the slope of the regression line grows with 
dimension $n$ in both cases.

Finally, we noticed that skipping the additional step in Algorithm~\ref{alg:AlphasqrLineSearch}
slightly reduces the computational effort in high dimensions for the model from 
\cite{ZJB2016} but without a real difference, while our model becomes slightly 
slower, see Figure 
\ref{Fig:test2}. The corresponding values are also given in 
Table \ref{Tab:runtime_without_add_step}. 
Table \ref{tab1} shows the number of line search updates required with and 
without the additional step,
and Table \ref{tab2} the number of functional evaluations. Note that in our 
implementation the evaluation of the gradient, which is only necessary one time 
in each outer iteration $k$ in the algorithm, is a more costly operation than a 
functional evaluation in a line search update.

\begin{figure}[H]
\begin{center}
\begin{tabular}{ccc}
\includegraphics[width=40mm, 
scale=0.35]{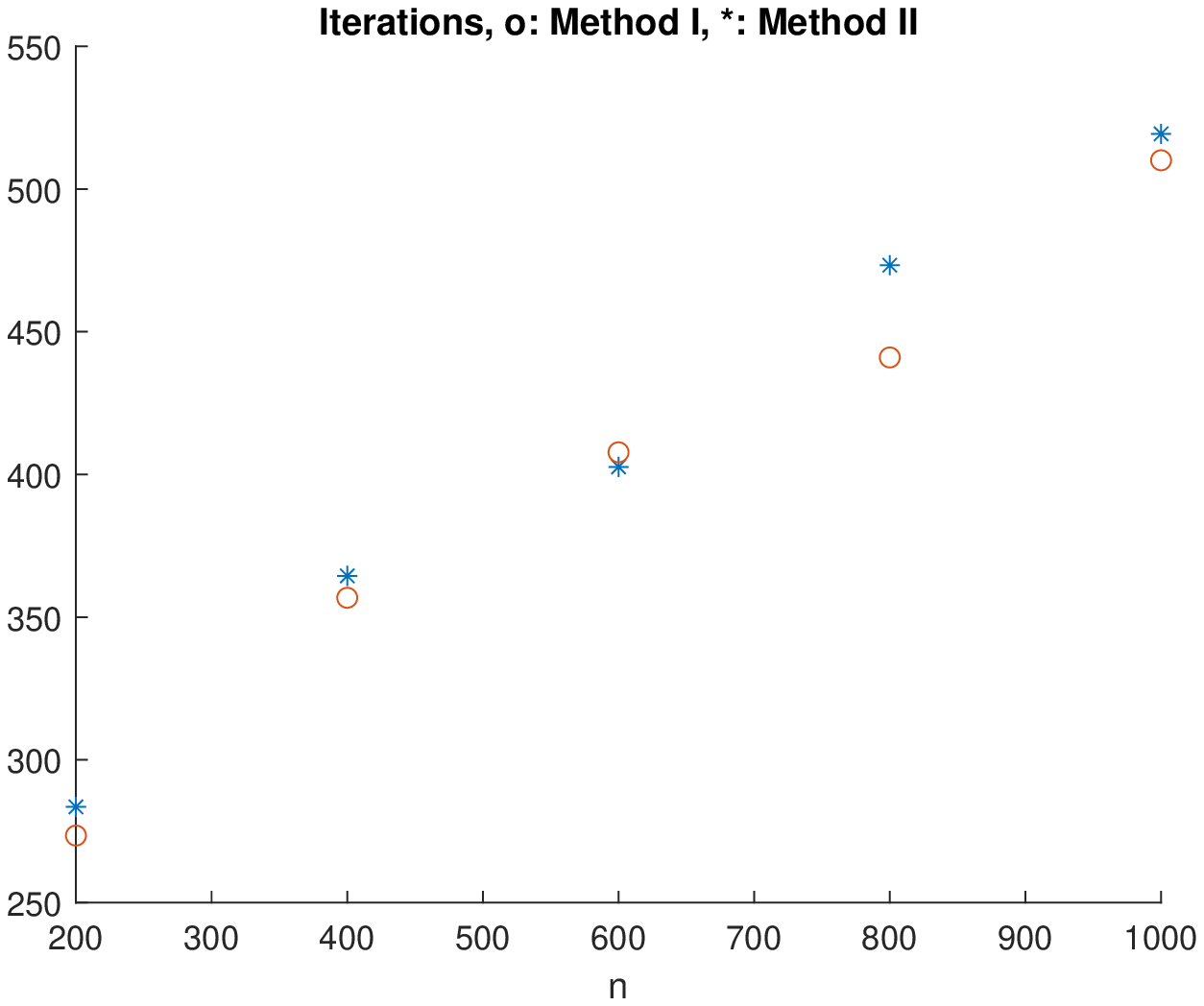} & 
\includegraphics[width=40mm, 
scale=0.35]{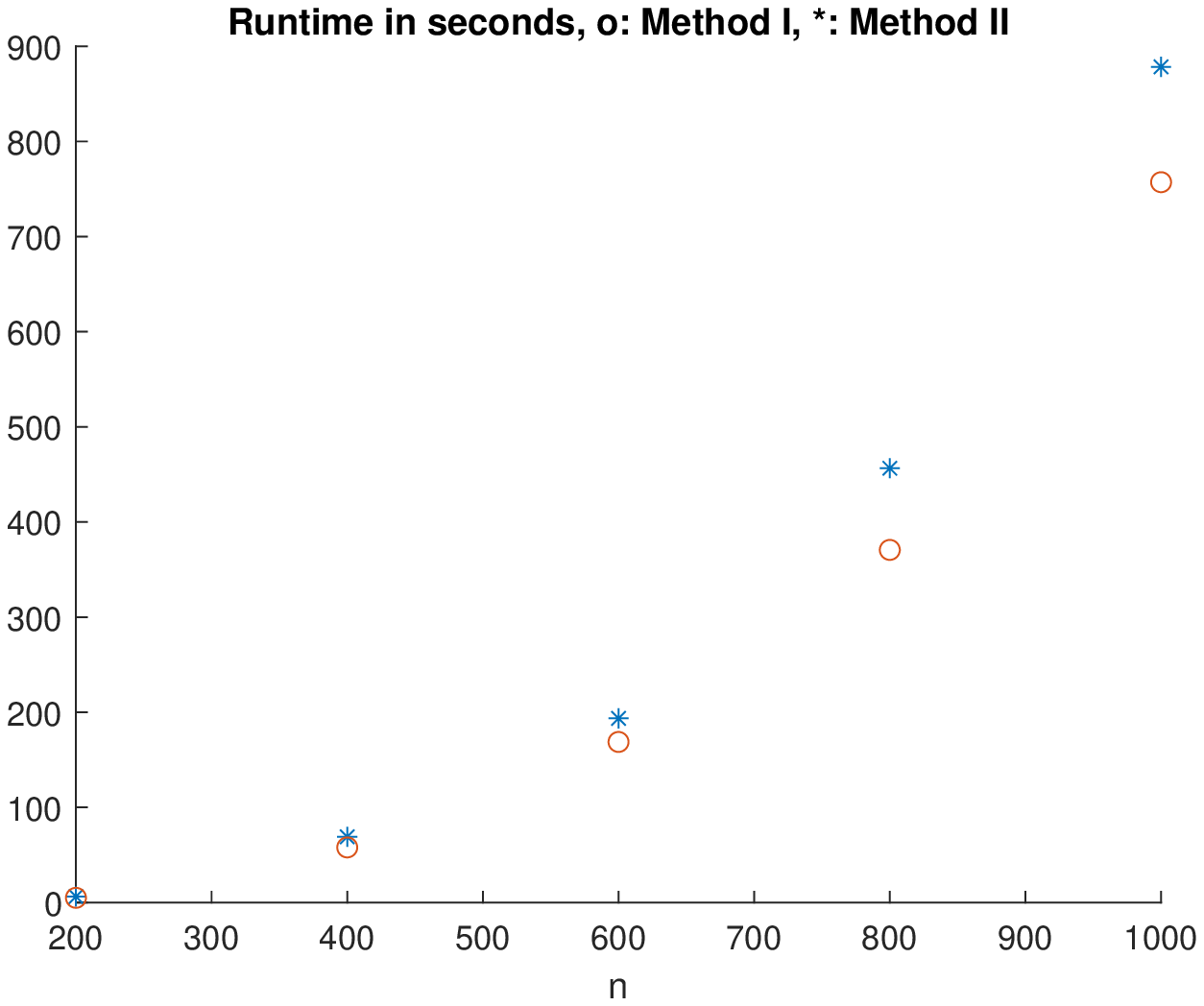} &
\includegraphics[width=40mm, 
scale=0.35]{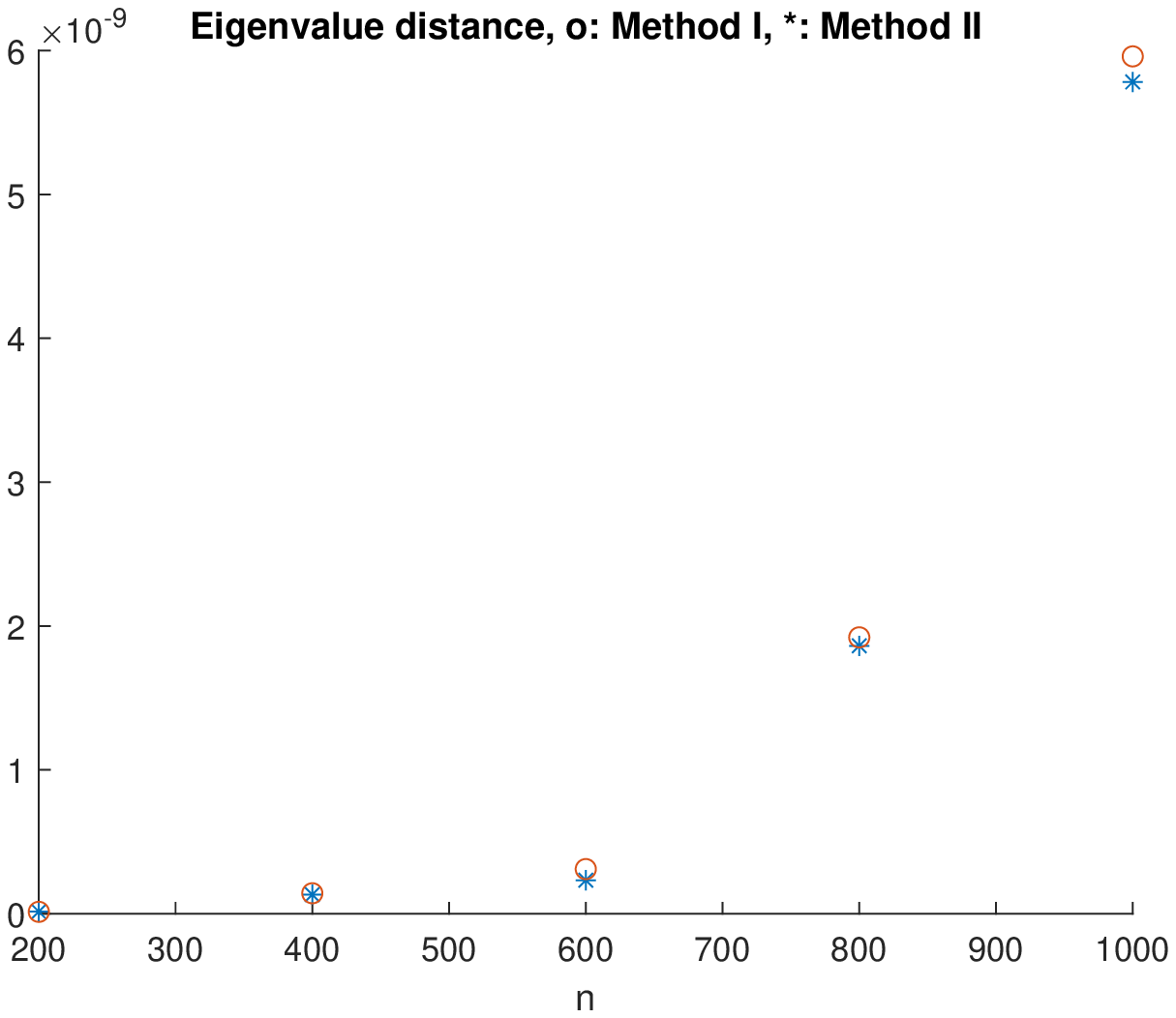}
\\
Number of iterations & Runtime & Eigenvalue distance
\end{tabular}
\end{center}
\caption{Comparison of Methods I and II.}
\label{Fig:test1}    
\end{figure}

\begin{table}[H]
\centering
\begin{tabular}{|l|c|c|c|c|c|}
\hline
 & $n=200$ & $n=400$ & $n=600$ & $n=800$ & $n=1000$ \\
 \hline
 Iterations, I & $273.5\pm 5.6$ & $356.8\pm 8.7$ & $407.8\pm 9.4$ & 
 $441\pm 10.3$ & $510.1\pm 14.2$ \\
 \hline
 Iterations, II & $283.5\pm 6.0$ & $364.5\pm 8.3$ & $402.6\pm 10.6$ & 
 $473.3\pm 10.9$ & $519.3\pm 13.2$ \\
 \hline
Runtime, I & $4.9 \pm 0.1$ & $57.9 \pm 0.6$ & $168.8 \pm 1.6$ & $370.7 \pm 3.6$ 
& 
$757.2 \pm 9.7$ \\
\hline
Runtime, II & $6.2\pm 0.1$ & $69.0\pm 1.4$ & $193.6 \pm 3.9$ & $456.5\pm 9.7$ & 
$878.4\pm 18.5$ \\
\hline
\end{tabular}
\caption{Iterations $k$ and Runtime in $s$ together with std error of 
the mean for $50$ 
samples.}
\label{Tab:runtime_with_add_step}
\end{table}

\begin{figure}
\begin{tabular}{ccc}
\includegraphics[width=40mm, scale=0.5]{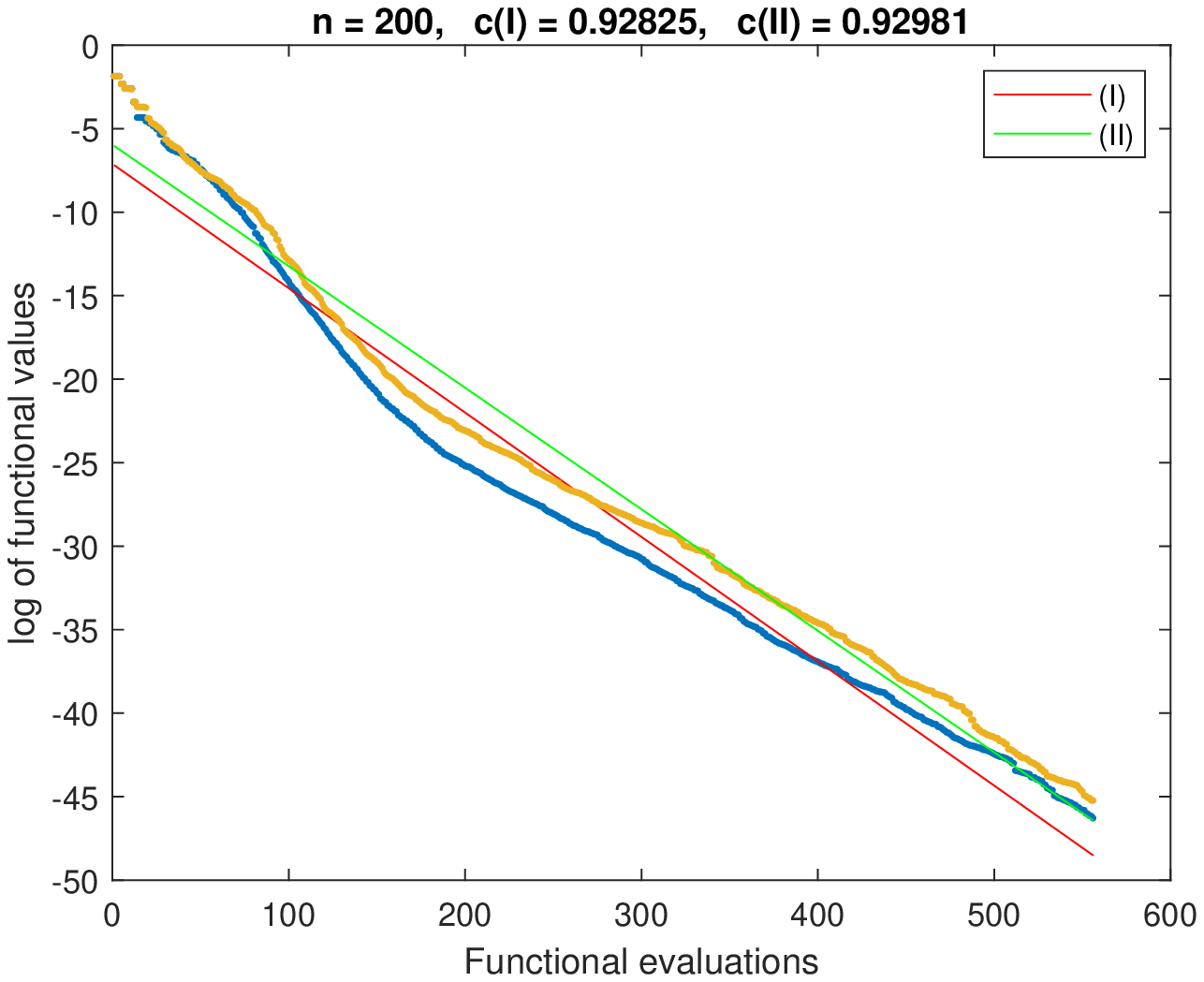} &
\includegraphics[width=40mm, scale=0.5]{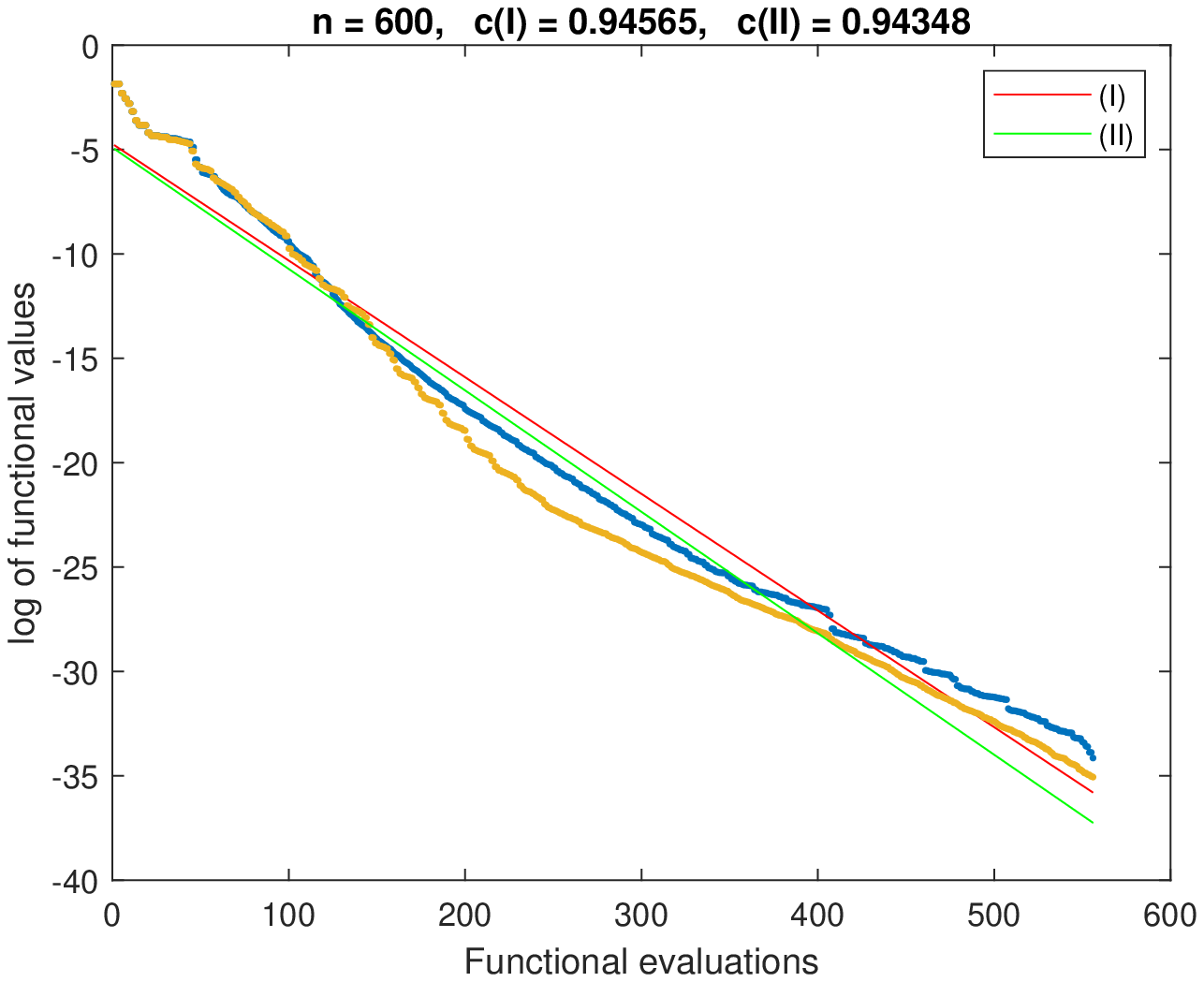} & 
\includegraphics[width=40mm, scale=0.5]{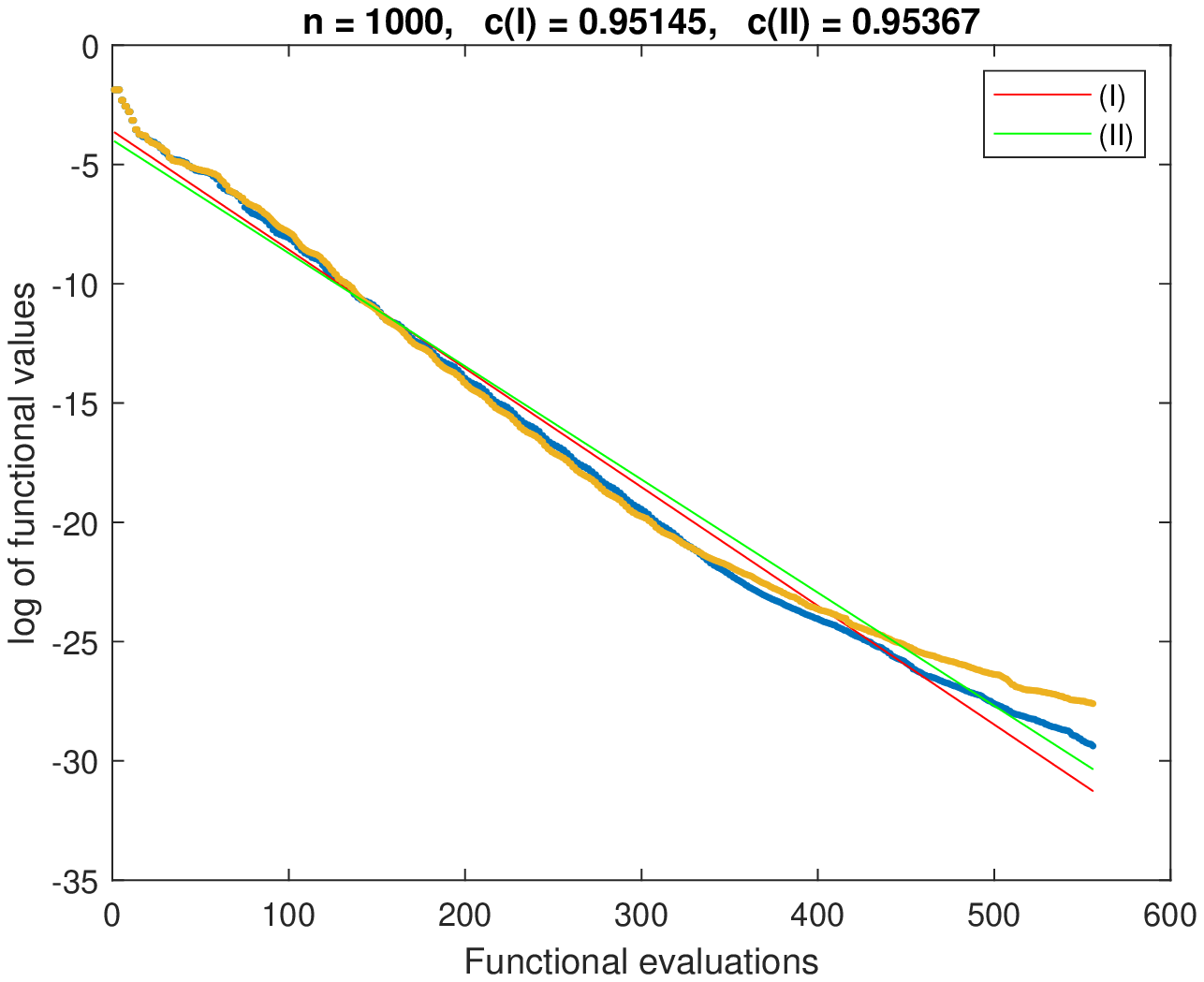}
\end{tabular}
\caption{Comparison of the convergence rates of Methods I (blue, red line) and II (yellow, green line) w.r.t. functional evaluations.}
\label{fig:ConvergenceRates}
\end{figure}

\begin{table}[H] 
\centering
\begin{tabular}{|c|c|c|c|c|c|}
\hline
 & $n=200$ & $n=400$ & $n=600$ & $n=800$ & $n=1000$ \\
 \hline
I & 120.6 & 148.9 & 180.3 & 194.2 & 203.8 \\
\hline
II & 139.9 & 176.1 & 200.5 & 231.2 & 245.9 \\
\hline
I without add. step & 26.8 & 33.0 & 40.7 & 44.6 & 56.6 \\
\hline
II without add step & 52.9 & 60.6 & 66.5 & 69.6 & 73.9 \\
\hline
\end{tabular}
\caption{Line search updates in Algorithm 
\ref{alg:AlphasqrLineSearch}.}\label{tab1}
\end{table}

\begin{table}[H] 
\centering
\begin{tabular}{|c|c|c|c|c|c|}
\hline
 & $n=200$ & $n=400$ & $n=600$ & $n=800$ & $n=1000$ \\
 \hline
I & 658.52 & 867.78 & 1001.8 & 1091.7 & 1234 \\
\hline
II & 670 & 869.4 & 972.3 & 1141.1 & 1243.1 \\
\hline
I without add. step & 334.1 & 469.5 & 582.3 & 702.8 & 807.1 \\
\hline
II without add. step & 371.3 & 508.1 & 593.2 & 657.6 & 748.5 \\
\hline
\end{tabular}
\caption{Number of functional evaluations in Algorithm \ref{alg:AlphasqrLineSearch}.}\label{tab2}
\end{table}

\begin{figure}[H]
\begin{center}
\begin{tabular}{ccc}
\includegraphics[width=40mm, scale=0.5]{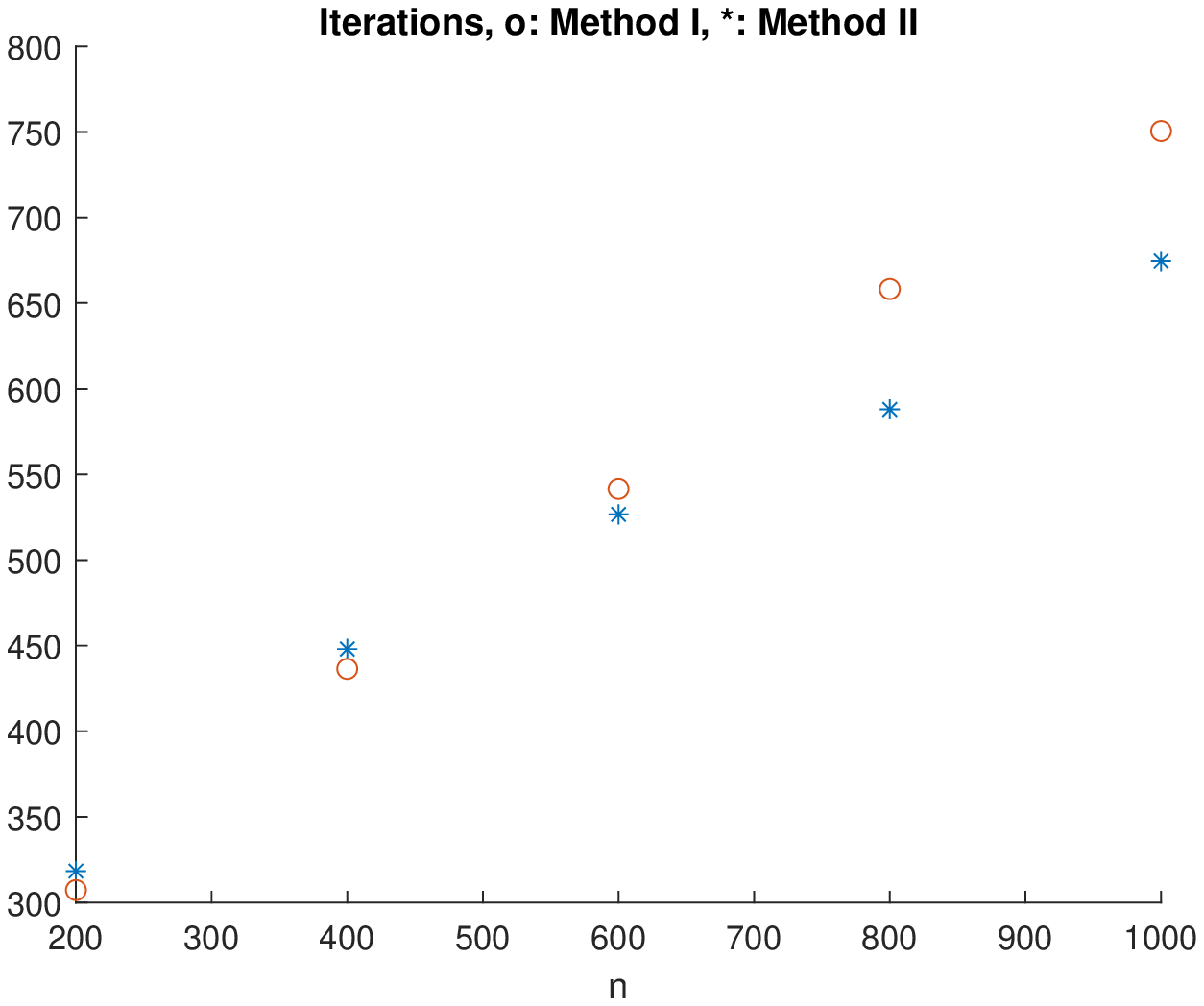} 
& 
\includegraphics[width=40mm, scale=0.5]{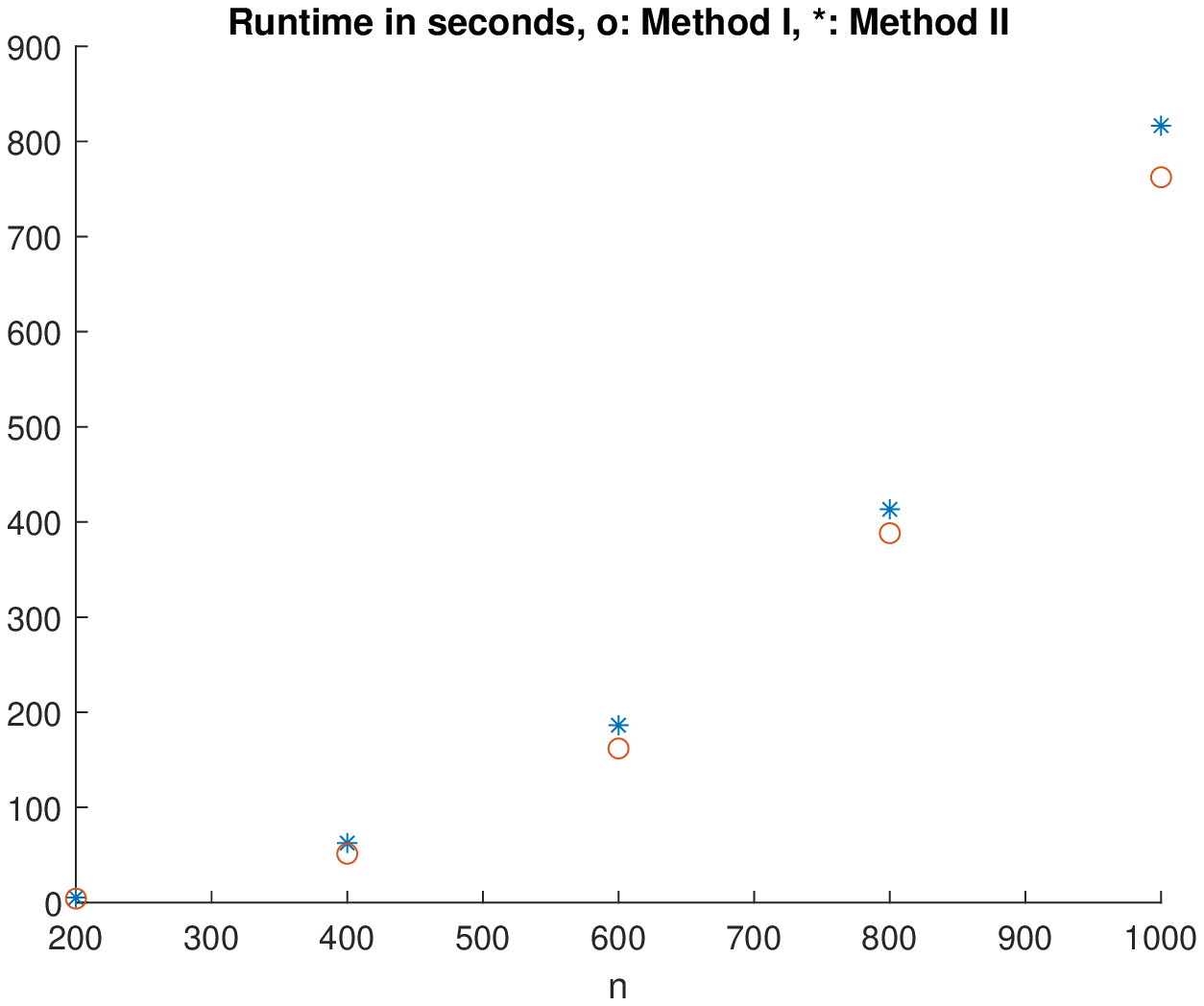} &
\includegraphics[width=40mm, 
scale=0.5]{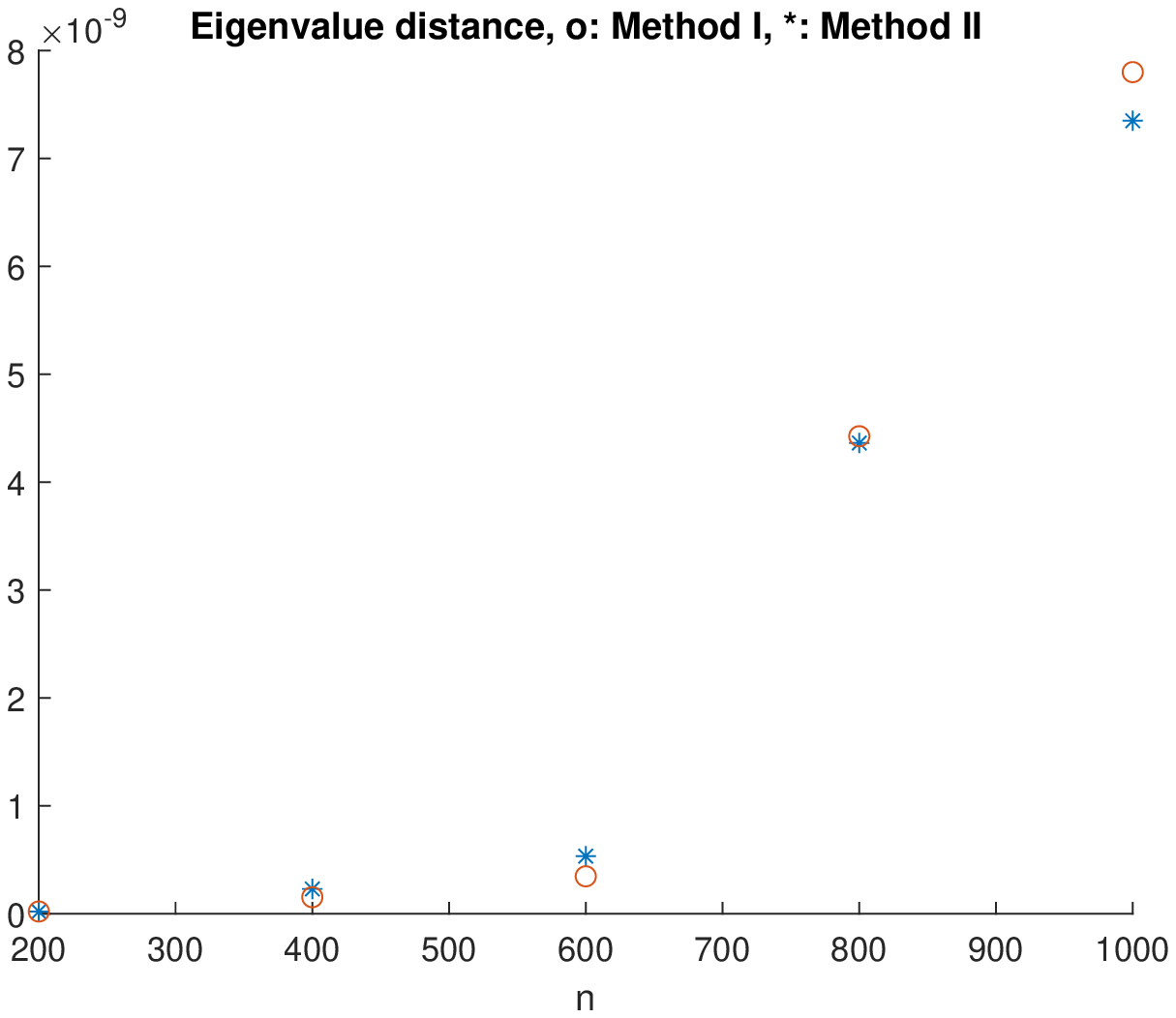}
\\
\textup{Number of iterations} & \textup{Runtime} &Eigenvalue distance
\end{tabular}
\end{center}
\caption{Comparison of Methods I and II if the line search without additional 
step is applied.}
\label{Fig:test2}    
\end{figure}

\begin{table}[H]
	\centering
	\begin{tabular}{|l|c|c|c|c|c|}
		\hline
		& $n=200$ & $n=400$ & $n=600$ & $n=800$ & $n=1000$ \\
		\hline
		Iterations, I & $307.3\pm 2.8$ & $436.5\pm 3.6$ & $541.6\pm 5.5$ & 
		$658.2\pm 9.8$ & $750.5\pm 14.3$ \\
		\hline
		Iterations, II & $318.4\pm 6.2$ & $448.0\pm 15.2$ & $526.7\pm 6.9$ & 
		$587.9\pm 6.4$ & $674.6\pm 9.7$ \\
		\hline
		Runtime, I & $4.1 \pm 0.0$ & $51.2 \pm 0.3$ & $161.9 \pm 1.4$ & $388.2 
		\pm 5.1$ 
		& 
		$762.4 \pm 12.7$ \\
		\hline
		Runtime, II & $5.2 \pm 0.1$ & $62.5\pm 2.0$ & $186.4 \pm 2.0$ & 
		$413.4\pm 3.7$ & 
		$816.6\pm 11.0$ \\
		\hline
	\end{tabular}
	\caption{Iterations $k$ and Runtime in $s$ together with std error of 
		the mean for $50$ samples if the line search without additional step is 
		applied.}
	\label{Tab:runtime_without_add_step}
\end{table}

\paragraph{Example with prescribed numbers of complex eigenvalues.} 
For $t \in \{ 1,\dots, \lfloor{n/2}\rfloor \}$, $n\geq 3$, 
we want to prescribe 
$t$ complex conjugate and 
$s$ real eigenvalues of a stochastic matrix. 
The theorem of Karpelevic \cite{Ka51} determines only the set of points 
$\Theta_n \subset \mathbb C$ that are eigenvalues of any
$n\times n$ stochastic matrix without characterizing the relation the eigenvalues of one stochastic matrix have to fulfill.
In other words, not every self-conjugate set with entries from $\Theta_n$ is the spectrum of a stochastic matrix.
We will make use of the following observation.

\begin{proposition}
Let $\Lambda = (1,\lambda_2,\dots,\lambda_{n})^\tT$ be a self-conjugate vector such that 
\begin{align*}
|\lambda_k| \leq \frac{1}{2n}
\end{align*}
for all $k=2,\dots,n$. Then there exists $A\in\mathcal{S}(n)$ whose eigenvalues are the components of $\Lambda$.
\end{proposition} 

\begin{proof}
We have $\Theta_m\subseteq \Theta_{m+1}$ for all $m\in\N$, see \cite{JPMP18}.
By Corollary \ref{CorNonnegToStoch}, it follows for $n \geq 3$ and all 
$k=2,\ldots,n$ that
\begin{align*}
\lambda_k \in B_\frac{1}{2n} \subseteq B_\frac{1}{2} \subseteq \Theta_3 \subseteq \Theta_n,
\end{align*}
where $B_r\coloneqq \{z\in\C:\ |z|\leq r\}$.
By~\cite[Corollary 4.18]{CG05} and \cite[p. 98]{CG05}, there exists $A\in\mathcal{S}(n)$ with eigenvalues from $\Lambda$ 
if $2n \max_{k=2,\dots,n} |\lambda_k|  \le 1$, which is just fulfilled for our 
setting.
\end{proof}

We choose $s-1$ real eigenvalues according to the uniform 
distribution on the interval \\
$[-1/(2n),1/(2n)]$ and $t$ pairs of conjugate complex eigenvalues according to 
the uniform distribution on $B_{1/(2n)}$ in the complex plane, which we 
simulate by
\begin{align*}
x,y = \texttt{randn}, \ \ \ \lambda = \frac{(x \pm iy) \cdot 
\sqrt{\texttt{rand}}}{\sqrt{x^2+y^2}},
\end{align*}
 cf. \cite[Algorithm 2.5.4]{RK16}.

Then, by the above proposition, there exists a stochastic matrix with these 
eigenvalues.
For larger $n$ the $n-1$ eigenvalues $\not = 1$ become rather small and problem 
(StIEP) is severely ill-posed.
Hence, we only consider a small size of $n=20$.
We observed that for small $t$ the distance to the input eigenvalues 
oscillates although the value of the functional decreases, while for larger 
$t$, these oscillations became negligible.
Therefore we stop the experiment after $3000$ iterations and choose those matrix having the smallest eigenvalue
distance from the given ones. Table
\ref{tab:TableVarying_t_min_dist} 
shows the minimal distance together with  the corresponding iteration number 
$k$ for various $t$ averaged over $200$ samples, respectively.

We observe that, as $t$ gets larger, better results are achieved. The iteration 
number and the eigenvalue distance appear to be rather independent of $t$ and 
the chosen method. The corresponding values from the stopping criterion are 
between $10^{-7}$ 
and $10^{-8}$, in particular the algorithms reach the minimal eigenvalue 
distance before they would have been stopped in the previous example.

Using the exponential maps as retractions we see that the minimal eigenvalue 
distances 
are attained much earlier and that the eigenvalues become 
slightly less close, 
see Table~\ref{tab:TableVarying_t_min_distExpMaps}. 

If the different retraction apart form the exponential map is chosen, both 
 methods start to run into problems with 
descending, but only for values smaller than approx. $10^{-15}$ in the stopping 
criterion, i.e. much later than it is usually stopped, which is supposed to be 
due to rounding errors. Interestingly, 
this 
behaviour is never observed, if the exponential maps are chosen as retractions. 
However, this is not an argument for choosing the exponential map, 
since the problems in descending occur at a point where basically rounding 
errors are minimized and we have observed that the choice of the exponential 
maps is costlier in total. Besides, we have seen that the chosen 
retractions different from the exponential map get closer to the eigenvalues in 
our experiment, see 
Table~\ref{tab:TableVarying_t_min_dist} and 
Table~\ref{tab:TableVarying_t_min_distExpMaps}.

\setlength\extrarowheight{2pt}

\begin{table}[H]
\begin{tabular}{|l|c|c|c|c|c|}
\hline 
$t$ & 1 & 2 & 3 & 4 & 5 \\
\hline
I: eigenvalue distance & $4.0\cdot 10^{-6}$ & $4.1\cdot 10^{-6}$ & $3.9\cdot 
10^{-6}$ & $1.2\cdot 10^{-7}$ & $8.7\cdot 10^{-9}$ \\
\hline
I: iteration $k$  & 1548.9 & 1651.3 & 1506 & 1539.3 & 1607.1 \\
\hline 
II: eigenvalue distance & $3.0\cdot 10^{-6}$ & $7.7\cdot 10^{-6}$ & $3.0\cdot 
10^{-6}$& $1.4\cdot 10^{-7}$ & $7.3\cdot 10^{-9}$ \\
\hline
II: iteration $k$  & 1666.2 & 1617.5 & 1676.9 & 1512.1 & 1647.5 \\
\hline
\end{tabular}
\end{table}

\begin{table}[H]
\begin{tabular}{|l|c|c|c|c|}
\hline
$t$ & 6 & 7 & 8 & 9 \\
\hline
I: eigenvalue distance& $2.7\cdot 10^{-10}$ & $3.1\cdot 10^{-11}$ & $7.5\cdot 
10^{-12}$ & $9.0\cdot 10^{-13}$ \\
\hline
I: iteration $k$ & 1634 & 1479.2 & 1580.5 & 1598.6 \\
\hline
II: eigenvalue distance  & $2.7\cdot 10^{-10}$ & $3.1\cdot 10^{-11}$ & 
$7.0\cdot 10^{-12}$ & $9.8\cdot 10^{-13}$ \\
\hline
II: iteration $k$ & 1625.8 & 1551.5 & 1618.2 & 1587.7 \\
\hline
\end{tabular}
\caption{Minimal eigenvalue distances and iteration number $k \le 3000$, where it is achieved.}
\label{tab:TableVarying_t_min_dist}
\end{table}

\setlength\extrarowheight{2pt}

\begin{table}[H]
	\begin{tabular}{|l|c|c|c|c|c|}
		\hline 
		$t$ & 1 & 2 & 3 & 4 & 5 \\
		\hline
		I: eigenvalue distance & $4.2\cdot 10^{-5}$ & $4.0\cdot 10^{-5}$ & 
		$1.4\cdot 10^{-5}$ & $2.7\cdot 10^{-6}$ & $5.2\cdot 10^{-7}$ \\
		\hline
		I: iteration $k$  & 156.3 & 160.2 & 160.0 & 161.3 & 160.4 \\
		\hline 
		II: eigenvalue distance & $3.4\cdot 10^{-5}$ & $4.3\cdot 10^{-5}$ & 
		$1.5\cdot 
		10^{-5}$& $2.4\cdot 10^{-6}$ & $9.7\cdot 10^{-7}$ \\
		\hline
		II: iteration $k$  & 161.3 & 161.2 & 162.1 & 160.4 & 164.2 \\
		\hline
	\end{tabular}
\end{table}

\begin{table}[H]
	\begin{tabular}{|l|c|c|c|c|}
		\hline
		$t$ & 6 & 7 & 8 & 9 \\
		\hline
		I: eigenvalue distance& $1.2\cdot 10^{-8}$ & $4.4\cdot 10^{-10}$ & 
		$1.3\cdot 
		10^{-10}$ & $1.0\cdot 10^{-11}$ \\
		\hline
		I: iteration $k$ & 160.5 & 159.3 & 159.8 & 162.0 \\
		\hline
		II: eigenvalue distance  & $7.8\cdot 10^{-9}$ & $6.2\cdot 10^{-10}$ & 
		$7.4\cdot 10^{-11}$ & $9.9\cdot 10^{-12}$ \\
		\hline
		II: iteration $k$ & 162.3 & 161.5 & 160.3 & 158.8 \\
		\hline
	\end{tabular}
	\caption{Corresponding results as in Table 
	\ref{tab:TableVarying_t_min_dist} if the exponential maps as retraction are chosen.}
	\label{tab:TableVarying_t_min_distExpMaps}
\end{table}

\section*{Appendix} \label{app}

The following proof of Theorem \ref{thm:SameConvergenceTheorem} is along the lines of \cite{ZJB2016}.
Throughout this section, let $x^{(k)} = (S^{(k)},Q^{(k)},V^{(k)},a^{(k)},b^{(k)}) \in \mathcal M_t$,
$d^{(k)}$, $y^{(k)}$ be the iterates generated by Algorithm~\ref{alg:cg_1} with $\alpha^{(k)}$ as in Algorithm~\ref{alg:AlphasqrLineSearch}.
Further, let
$$
{\mathcal Z} \coloneqq \mathrm{lev}_{F(x^{(0)})} F.
$$
By Proposition \ref{lem:ExistenceMinimizers}, $\mathcal Z \subset \mathcal{M}_t$ is compact
and $x^{(k)} \in \mathcal{Z}$ for all $k \in \mathbb N$.
The proof is based on various auxiliary lemmas.


\begin{lemma}\label{lem:ConvergenceOfAlphakdk}
	It holds
	\begin{align}
	\lim_{k\to\infty} \alpha^{(k)}  \bigl \Vert d^{(k)} \bigr \Vert_{x^{(k)}} = 0.
	\end{align}
\end{lemma}

\begin{proof}
	The line search in Algorithm~\ref{alg:AlphasqrLineSearch} ensures that 
	\begin{align*}
	\delta \sum_{k=0}^\infty {\alpha^{(k)}}^2{\bigl\| d^{(k)}}\bigr\|_{x^{(k)}}^2 
	\leq 
	\sum_{k=0}^\infty \Bigl(F\bigl(x^{(k)}\bigr)-F\bigl(x^{(k+1)}\bigr)\Bigr) = F\bigl(x^{(0)}\bigr) - \inf_{k\in\N} F\bigl(x^{(k)}\bigr) < \infty,
	\end{align*}
	where the last inequality follows since $\inf_{x \in \mathcal Z} F(x)$ is finite by Proposition~\ref{lem:ExistenceMinimizers}.
\end{proof}

\begin{lemma} \label{prop:YLeqSearchDirection}
	There exists a constant $C>0$ such that for all $k$ sufficiently large, 
	\begin{align*}
	\bigl \Vert y^{(k)} \bigr \Vert_{x^{(k+1)}} \leq C \alpha^{(k)}\bigl \Vert d^{(k)} \bigr \Vert_{x^{(k)}}.
	\end{align*}
\end{lemma}

\begin{proof} 
	By the applied vector transport \eqref{eqn:ProjTransport} and since the projection is nonexpansive, we obtain
	\begin{align*}
	&\bigl \| y^{(k)} \bigr\|_{ x^{(k+1)} }^2 \\
	&= \bigl\| g^{(k+1)} - {\mathcal T}_{ x^{(k)}, \alpha^{(k)} d^{(k)} }  g^{(k)}\bigr \|_{ x^{(k+1)} }^2
	= \bigl \| g^{(k+1)} - \Pi_{T_{x^{(k+1)}} \mathcal M_t} g^{(k)} \bigr\|_{x^{(k+1)} }^2\\
	&\le
	\bigl \| \nabla_{S,Q,V,b} F \bigl(x^{(k+1)}\bigr) - \nabla_{S,Q,V,b} F \bigl(x^{(k)}\bigr)\bigr\|^2 
	+
	\bigl\| \nabla_{a} F \bigl(x^{(k+1)}\bigr) - \nabla_{a} F \bigl(x^{(k)}\bigr)\bigr\|_{a^{(k+1)}}^2 \\
	&\le
	\bigl\| \nabla_{S,Q,V,b}^E F \bigl(x^{(k+1)}\bigr) -  \nabla_{S,Q,V,b}^E F \bigl(x^{(k)}\bigr)\bigr\|^2
	+
	\bigl( a^{(k+1)} \bigr)^{-2} \bigl\| \nabla_{a} F \bigl(x^{(k+1)}\bigr) - \nabla_{a} F \bigl(x^{(k)}\bigr)\bigr\|^2,
	\end{align*}
	where $\nabla_{S,Q,V,b}^E$ denotes the Euclidean gradient on $\left(\mathbb R^{n,n} \right)^3 \times \mathbb R^t$.
	By Lemma \ref{gradients}, the gradients are continuously differentiable on $\mathcal M_t$
	such that they are Lipschitz continuous on $\mathcal Z$. Hence there exists $C >0$ with
	$$
	\bigl\| y^{(k)}\bigr\|_{ x^{(k+1)} }^2 
	\le 
	C \mathrm{dist}_{\mathcal M_t}^2 \bigl(x^{(k+1)}, x^{(k)} \bigr)
	=
	C \mathrm{dist}_{\mathcal M_t}^2 \bigl(\mathcal R_{x^{(k)}}\bigl(\alpha^{(k)} d^{(k)}\bigr), x^{(k)}\bigr),
	$$
	Since ${\mathcal Z}$ is compact and by using Lemma~\ref{lem:ConvergenceOfAlphakdk} together with \cite[Proposition~7.4.5]{AMS09}, we finally obtain
	\[\bigl\| y^{(k)} \bigr\|_{ x^{(k+1)} } \leq C \alpha^{(k)} \bigl\|d^{(k)}\bigr\|_{x^{(k)}}.\]
	
\end{proof}

\begin{lemma}\label{lem:dBounded}
	Suppose that there exists $\epsilon>0$ such that 
	\begin{align}
	\label{eqn:GradGeqEpsilon}
	\bigl \Vert \nabla_{\mathcal{M}_t} F\bigl(x^{(k)}\bigr) \bigr \Vert_{x^{(k)}} \geq \epsilon
	\end{align}
	for all $k\in\N$.
	Then there exists $C>0$ such that for all $k\in\N$ it holds
	\begin{align*}
	\bigl \Vert d^{(k)}\bigr \Vert_{x^{(k)}} \leq C.
	\end{align*}
\end{lemma}

\begin{proof}
	First, by compactness of $\mathcal{Z}$ and continuity of $\nabla_{\mathcal{M}_t} F$ on $\mathcal{Z}$, there exists $\gamma>0$ such that
	\begin{align}\label{eqn:GradBounded}
	\| \nabla_{\mathcal{M}_t}F(x) \|_x
	\leq \gamma 
	\end{align}
	for all $x\in\mathcal{Z}$.
	Since the $a$ component of $\mathcal Z$ is uniformly bounded away from zero, it holds for $x \in \mathcal{Z}$, $d\in T_x \mathcal{M}_t$, $\xi \in T_x \mathcal{M}_t$ and $z = \mathcal R_x(d) \in \mathcal Z$ that
	\begin{align}\label{eqn:TransportEstimation}
	\norm{\mathcal{T}_{x,d}\xi}_z  = \norm{\Pi_{T_z \mathcal{M}_t} (\xi)}_z \leq \norm{\xi}_z \leq C \norm{\xi}_x, \qquad C > 0,
	\end{align}
	and we can estimate $\vert \theta^{(k)} \vert $ in Algorithm~\ref{alg:cg_1} as follows 
	\begin{align} \label{eqn:ThetaEstimation}
	\bigl \vert \theta^{(k)} \bigr \vert 
	&= 
	\frac{\bigl|\bigl\langle\nabla_{\mathcal{M}_t} F\bigl(x^{(k+1)}\bigr),\mathcal{T}_{x^{(k)},\alpha^{(k)} d^{(k)}}d^{(k)}\bigr\rangle_{x^{(k+1)}} \bigr|}{\norm{\nabla_{\mathcal{M}_t}F\bigl(x^{(k)}\bigr)}_{x^{(k)}}^2}\\
	&\leq 
	C \bigl \Vert d^{(k)} \bigr \Vert_{x^{(k)}} \frac{\norm{\nabla_{\mathcal{M}_t}F\bigl(x^{(k+1)}\bigr)}_{x^{(k+1)}} }{\norm{\nabla_{\mathcal{M}_t}F\bigl(x^{(k)}\bigr)}^2_{x^{(k)}}}.
	\end{align}
	Similarly, we obtain for $\beta^{(k)}$ using Lemma~\ref{prop:YLeqSearchDirection} that there exists an integer $k_1>0$ such that for all $k\geq k_1$ it holds
	\begin{align}\label{eqn:BetaEstimation}
	\bigl \vert \beta^{(k)} \bigr \vert \leq 
	C \alpha^{(k)} \bigl \Vert d^{(k)} \bigr \Vert_{x^{(k)}} \frac{\norm{\nabla_{\mathcal{M}_t}F\bigl(x^{(k+1)}\bigr)}_{x^{(k+1)}}}{\norm{\nabla_{\mathcal{M}_t}F\bigl(x^{(k)}\bigr)}^2_{x^{(k)}}}.
	\end{align}
	Using the definition of $d^{(k+1)}$ together with \eqref{eqn:GradBounded}, \eqref{eqn:TransportEstimation} and Lemma~\ref{prop:YLeqSearchDirection}, we get for all $k\geq k_1$ that
	\[\bigl \| d^{(k+1)} \bigr\|_{x^{(k+1)}} 
	\leq 
	\gamma 
	+ C \bigl \| d^{(k)} \bigr \|_{x^{(k)}} \Bigr(\bigl \vert \beta^{(k)} \bigr \vert
	+ \alpha^{(k)} \bigl \vert \theta^{(k)} \bigr \vert \Bigr).\]
	Next, plugging in \eqref{eqn:ThetaEstimation} and \eqref{eqn:BetaEstimation}
	gives
	\begin{align*}
	\bigl \| d^{(k+1)} \bigr\|_{x^{(k+1)}} 
	&\leq 
	\gamma 
	+ 2C \alpha^{(k)}\bigl\| d^{(k)}\bigr\|^2_{x^{(k)}} \frac{\bigl\| \nabla_{\mathcal{M}_t} F\bigl(x^{(k+1)}\bigr) \bigr\|_{x^{(k+1)}}}{\bigl\| \nabla_{\mathcal{M}_t} F\bigl(x^{(k)}\bigr)\bigr\|_{x^{(k)}}^2}
	\\
	&\leq 
	\gamma + 2C\frac{\gamma}{\epsilon^2}\alpha^{(k)} \bigl \Vert d^{(k)} \bigr \Vert_{x^{(k)}}^2 .
	\end{align*}
	By Lemma \ref{lem:ConvergenceOfAlphakdk}, there exist $r\in (0,1)$ and $k_2>0$ such that for all $k\geq k_2$, 
	\begin{align*}
	2C\frac{\gamma}{\epsilon^2} \alpha^{(k)} \|d^{(k)}\|_{x^{(k)}} \leq r.
	\end{align*}
	Then, we conclude for all $k\geq k_0 \coloneqq \max(k_1,k_2)$ that
	\begin{align*}
	\bigl \|d^{(k+1)}\bigr \|_{x^{(k+1)}}  &\leq \gamma + r\bigl \Vert d^{(k)} \bigr \Vert_{x^{(k)}} \leq \gamma\sum_{i=0}^{k-k_0}r^i + r^{k-k_0+1}\bigr \Vert d^{(k_0)} \bigl \Vert_{x^{(k_0)}} \\
	&\leq \frac{\gamma}{1-r} + \bigl \Vert d^{(k_0)} \bigr \Vert_{x^{(k_0)}}.
	\end{align*}
	Finally, we estimate
	\begin{align*}
	\bigl \| d^{(k)}\bigr \|_{x^{(k)}} \leq  \frac{\gamma}{1-r} + \max_{1 \leq i \leq k_0} \bigl \Vert d^{(i)}\bigr \Vert_{x^{(i)}}.
	\end{align*}
\end{proof}

Now, we are able to prove Theorem~\ref{thm:SameConvergenceTheorem}.
\\[2ex]
\begin{proof}[Proof of Theorem \textup{\ref{thm:SameConvergenceTheorem}}:]
For the sake of contradiction, we assume that there exists a constant $\epsilon>0$ 
such that 
$\|\nabla_{\mathcal{M}_t} F(x^{(k)})\|_{x^{(k)}} \geq \epsilon$ for all $k\in\N$. 
By construction of $d^{(k)}$ we get
\begin{align*}
\bigl \| \nabla_{\mathcal{M}_t}F\bigl(x^{(k)}\bigr)\bigr \|_{x^{(k)}}^2 
= 
-\Bigl \langle d^{(k)},\nabla_{\mathcal{M}_t}F\bigl(x^{(k)}\bigr) \Bigr \rangle_{x^{(k)}} 
\leq \bigl\| d^{(k)} \bigr\|_{x^{(k)}} \bigl \Vert \nabla_{\mathcal{M}_t}F\bigl(x^{(k)}\bigr) \bigr \Vert_{x^{(k)}},
\end{align*}
which implies $\Vert \nabla_\mathcal{M}F(x^{(k)}) \Vert_{x^{(k)}}\leq\Vert d^{(k)} \Vert_{x^{(k)}}$. 
Combining this with Lemma \ref{lem:ConvergenceOfAlphakdk}, we get 
\begin{align*}
\lim_{k\to\infty} \alpha^{(k)}\bigl \Vert \nabla_{\mathcal{M}_t}F\bigl(x^{(k)}\bigr) \bigr \Vert_{x^{(k)}} = 0
\end{align*}
and consequently also $\lim_{k\to\infty}\alpha^{(k)}=0$. 

It follows from the line search condition in 
Algorithm~\ref{alg:AlphasqrLineSearch} that
\begin{align}\label{eqn:1}
F \circ \mathcal{R}_{x^{(k)}}\bigl(\tau^{-1} \alpha^{(k)}d^{(k)}\bigr) - F\bigl(x^{(k)}\bigr) 
\geq 
-\delta\tau^{-2}{\alpha^{(k)}}^2\bigl \Vert d^{(k)} \bigr \Vert_{x^{(k)}}^2.
\end{align}
Next, we consider the so-called pullback function $\hat{F}\coloneqq F\circ\mathcal{R}\colon T \mathcal{M}_t\rightarrow\R$, which is $C^\infty$ as concatenation of $C^\infty$ functions and fulfills
\begin{align}
\label{eqn:GradientPullbackInZero}
\nabla\hat{F}_x(0) = \nabla_{\mathcal{M}_t} F(x), \ \ \ x \in {\mathcal{M}_t},
\end{align}
due to the properties of retractions. For $x\in \mathcal{M}_t$, we denote the restriction of $\hat{F}$ to $T_x\mathcal{M}_t$ by $\hat{F}_x$.
Since $\hat F$ is $C^\infty$ and hence its gradient is Lipschitz on compact sets, there exists $L>0$ such that
\begin{align}
\label{eqn:PullBackLContGradient}
\bigl \Vert \nabla \hat{F}_x(\eta) - \nabla \hat{F}_x(0) \bigr \Vert_x \leq L \norm{\eta}_x
\end{align}
for all $x\in\mathcal{Z}$ and $\eta\in T_x\mathcal{M}_t$ with $\norm{\eta}_x\leq1$.

By the mean-value theorem we obtain for some $\omega^{(k)}\in(0,1)$ that
\begin{align*}
F \circ \mathcal{R}_{x^{(k)}}\bigl(\tau^{-1}\alpha^{(k)}d^{(k)}\bigr) - F\bigl(x^{(k)}\bigr) 
&= \hat{F}_{x^{(k)}}\bigl(\tau^{-1}\alpha^{(k)}d^{(k)}\bigr) - \hat{F}_{x^{(k)}}(0) \\
&= \tau^{-1}\alpha^{(k)}\Bigl\langle \nabla\hat{F}_{x^{(k)}} \bigl(\omega^{(k)}\tau^{-1}\alpha^{(k)}d^{(k)}\bigr), d^{(k)}\Bigr\rangle_{x^{(k)}}
\end{align*}
and further by \eqref{eqn:PullBackLContGradient} for $\Vert \alpha^{(k)}d^{(k)} \Vert_{x^{(k)}} \leq \tau$ that
\begin{align*}
&F\circ \mathcal{R}_{x^{(k)}}\bigl(\tau^{-1}\alpha^{(k)}d^{(k)}\bigr) - F\bigl(x^{(k)}\bigr)\\
&= \tau^{-1}\alpha^{(k)}\Bigl\langle \nabla\hat{F}_{x^{(k)}}(0) +  \nabla\hat{F}_{x^{(k)}}\bigl(\omega^{(k)}\tau^{-1}\alpha^{(k)}d^{(k)}\bigr) - \nabla\hat{F}_{x^{(k)}}(0),d^{(k)}\Bigr\rangle_{x^{(k)}}\\
&\leq \tau^{-1}\alpha^{(k)}\Bigl \langle\nabla\hat{F}_{x^{(k)}}(0),d^{(k)}\Bigr \rangle_{x^{(k)}} + L\omega^{(k)}\tau^{-2}{\alpha^{(k)}}^2\bigl \Vert d^{(k)} \bigr \Vert^2_{x^{(k)}}\\
&= \tau^{-1}\alpha^{(k)}\Bigl \langle\nabla_{\mathcal{M}_t}F\bigl(x^{(k)}\bigr),d^{(k)}\Bigr\rangle_{x^{(k)}} + L\omega^{(k)}\tau^{-2}{\alpha^{(k)}}^2\bigl \Vert d^{(k)}\bigr \Vert^2_{x^{(k)}}\\
&\leq -\tau^{-1}\alpha^{(k)}\bigl \Vert \nabla_{\mathcal{M}_t}F\bigl(x^{(k)}\bigr)\bigr \Vert^2_{x^{(k)}} + L\tau^{-2}{\alpha^{(k)}}^2\bigl \Vert d^{(k)} \bigl \Vert^2_{x^{(k)}}.
\end{align*}
Together with \eqref{eqn:1} we conclude for sufficiently large $k$ that
\begin{align*}
\bigl \Vert \nabla_{\mathcal{M}_t}F\bigl(x^{(k)}\bigr) \bigr \Vert_{x^{(k)}}^2 
&\leq -\frac{\tau}{\alpha^{(k)}}\left(F\circ \mathcal{R}_{x^{(k)}}\bigl(\tau^{-1}\alpha^{(k)}d^{(k)}\bigr) - F\bigl(x^{(k)}\bigr)\right) + L\tau^{-1}\alpha^{(k)}\bigr \Vert d^{(k)} \bigl \Vert^2_{x^{(k)}}\\
&\leq (L+\delta)\tau^{-1}\alpha^{(k)}\bigr \Vert d^{(k)} \bigr \Vert^2_{x^{(k)}}.
\end{align*}
Since $\{d^{(k)}\}_k$ is bounded by Lemma \ref{lem:dBounded} and $\lim_{k\to\infty}\alpha^{(k)}=0$, we obtain the contradiction
\begin{align*}
\lim_{k\to\infty} \bigl \Vert \nabla_{\mathcal{M}_t} F\bigl(x^{(k)}\bigr) \bigr \Vert_{x^{(k)}} = 0,
\end{align*}
which concludes the proof.
\end{proof}

\subsection*{Acknowledgments} 
The authors want to thank S. Neumayer (TU Berlin) for fruitful discussions.
Funding by the German Research Foundation (DFG) with\-in the project STE 571/16-1 is gratefully acknowledged.

\bibliographystyle{abbrv}
\bibliography{references_max}

\begin{thebibliography}{10}

\bibitem{AMS09}
P.-A. Absil, R.~Mahony, and R.~Sepulchre.
\newblock {\em Optimization Algorithms on Matrix Manifolds}.
\newblock Princeton University Press, 2009.

\bibitem{ALRL11}
E.~Andruchow, G.~Larotonda, L.~Recht, and A.~Varela.
\newblock The left invariant metric in the general linear group.
\newblock {\em Journal of Geometry and Physics}, 86:241--257, 2014.

\bibitem{APSS2016}
F.~{\AA}str{\"o}m, S.~Petra, B.~Schmitzer, and C.~Schn{\"o}rr.
\newblock Image labeling by assignment.
\newblock {\em Journal of Mathematical Imaging and Vision}, 58(2):211--238,
  2017.

\bibitem{BFPS16}
R.~Bergmann, J.~H. Fitschen, J.~Persch, and G.~Steidl.
\newblock Iterative multiplicative filters for data labeling.
\newblock {\em International Journal of Computer Vision}, 123(3):435--453,
  2017.

\bibitem{BNO03}
D.~Bertsekas and A.~Nedic.
\newblock {\em Convex Analysis and Optimization}.
\newblock Athena Scientific, 2003.

\bibitem{CGM19}
F.~Cacace, A.~Germani, and C.~Manes.
\newblock Karpelevich theorem and the positive realization of matrices.
\newblock In {\em 2019 IEEE 58th conference on decision and control (CDC)},
  pages 6074--6079. IEEE, 2019.

\bibitem{CG05}
M.~Chu and G.~Golub.
\newblock {\em Inverse Eigenvalue Problems: Theory, Algorithms, and
  Applications}, volume~13.
\newblock Oxford University Press, 2005.

\bibitem{CG98}
M.~Chu and Q.~Guo.
\newblock A numerical method for the inverse stochastic spectrum problem.
\newblock {\em SIAM Journal on Matrix Analysis and Applications},
  19(4):1027--1039, 1998.

\bibitem{Chu91}
M.~T. Chu and K.~R. Driessel.
\newblock Constructing symmetric nonnegative matrices with prescribed
  eigenvalues by differential equations.
\newblock {\em SIAM Journal on Mathematical Analysis}, 22(5):1372--1387, 1991.

\bibitem{CMMD14}
L.~Ciampolini, S.~Meignen, O.~Menut, and T.~David.
\newblock Direct solution of the inverse stochastic problem through elementary
  markov state disaggregation.
\newblock {\em Archive Ouverte}, 2014.

\bibitem{Da67}
J.~W. Daniel.
\newblock The conjugate gradient method for linear and nonlinear operator
  equations.
\newblock {\em SIAM Journal on Numerical Analysis}, 4(1):10--26, 1967.

\bibitem{ELN04}
P.~D. Egleston, T.~D. Lenker, and S.~K. Narayan.
\newblock The nonnegative inverse eigenvalue problem.
\newblock {\em Linear Algebra and its Applications}, 379:475--490, 2004.

\bibitem{HZ2005}
W.~W. Hager and H.~Zhang.
\newblock A survey of nonlinear conjugate gradient methods.
\newblock {\em Pacific Journal of Optimization}, 2(1):35--58, 2006.

\bibitem{HJ12}
R.~A. Horn and C.~R. Johnson.
\newblock {\em Matrix Analysis}.
\newblock Cambridge University Press, 2012.

\bibitem{HA17}
S.~Hosseini and A.~Uschmajew.
\newblock A {R}iemannian gradient sampling algorithm for nonsmooth optimization
  on manifolds.
\newblock {\em SIAM Journal on Optimization}, 27(1):173--189, 2017.

\bibitem{Ito97}
H.~Ito.
\newblock A new statement about the theorem determining the region of
  eigenvalues of stochastic matrices.
\newblock {\em Linear Algebra and its Applications}, 267:241--246, 1997.

\bibitem{JPMP18}
C.~R. Johnson, C.~Mariju{\'a}n, P.~Paparella, and M.~Pisonero.
\newblock The {NIEP}.
\newblock In {\em Operator Theory, Operator Algebras, and Matrix Theory}, pages
  199--220. Springer, 2018.

\bibitem{JP17}
C.~R. Johnson and P.~Paparella.
\newblock A matricial view of the {K}arpelevi{\v{c}} theorem.
\newblock {\em Linear Algebra and its Applications}, 520:1--15, 2017.

\bibitem{Ka51}
F.~I. Karpelevic.
\newblock On the characteristic roots of matrices with nonnegative elements.
\newblock {\em Izvestiya Rossiiskoi Akademii Nauk. Seriya Matematicheskaya},
  15(4):361--383, 1951.

\bibitem{Lin15}
M.~M. Lin.
\newblock An algorithm for constructing nonnegative matrices with prescribed
  real eigenvalues.
\newblock {\em Applied Mathematics and Computation}, 256:582--590, 2015.

\bibitem{LL78}
R.~Loewy and D.~London.
\newblock A note on an inverse problem for nonnegative matrices.
\newblock {\em Linear and Multilinear Algebra}, 6(1):83--90, 1978.

\bibitem{Mi88}
H.~Minc.
\newblock {\em Nonnegative Matrices}.
\newblock Wiley-Intersci. Ser. Discrete Math. Optim. John Wiley \& Sons, Inc.,
  1988.

\bibitem{Orsi06}
R.~Orsi.
\newblock Numerical methods for solving inverse eigenvalue problems for
  nonnegative matrices.
\newblock {\em SIAM Journal on Matrix Analysis and Applications},
  28(1):190--212, 2006.

\bibitem{Perfect53}
H.~Perfect.
\newblock Methods of constructing certain stochastic matrices.
\newblock {\em Duke Mathematical Journal}, 20(3):395--404, 1953.

\bibitem{PPST19}
G.~Plonka, D.~Potts, G.~Steidl, and M.~Tasche.
\newblock {\em Numerical Fourier Analysis}.
\newblock Springer, 2018.

\bibitem{Ro06}
W.~Rossmann.
\newblock {\em Lie Groups: An Introduction through Linear Groups}.
\newblock Oxford University Press, 2006.

\bibitem{RK16}
R.~Y. Rubinstein and D.~P. Kroese.
\newblock {\em Simulation and the Monte Carlo Method}, volume~10.
\newblock John Wiley \& Sons, 2016.

\bibitem{Sm94}
S.~T. Smith.
\newblock Optimization techniques on {R}iemannian manifolds.
\newblock {\em Fields Institute Communications}, 3(3):113--135, 1994.

\bibitem{Suleimanova49}
H.~Sule{\i}manova.
\newblock Stochastic matrices with real characteristic numbers.
\newblock In {\em Doklady Akad. Nauk SSSR (NS)}, volume~66, pages 343--345,
  1949.

\bibitem{Xu98}
S.~Xu.
\newblock {\em An Introduction to Inverse Algebraic Eigenvalue Problems}.
\newblock Peking University Press, 1998.

\bibitem{ZZL06}
L.~Zhang, W.~Zhou, and D.-H. Li.
\newblock A descent modified {P}olak--{R}ibi{\`e}re--{P}olyak conjugate
  gradient method and its global convergence.
\newblock {\em IMA Journal of Numerical Analysis}, 26(4):629--640, 2006.

\bibitem{ZJB2016}
Z.~Zhao, X.-Q. Jin, and Z.-J. Bai.
\newblock A geometric nonlinear conjugate gradient method for stochastic
  inverse eigenvalue problems.
\newblock {\em SIAM Journal on Numerical Analysis}, 54(4):2015--2035, 2016.

\end{thebibliography}

\end{document}